\documentclass[12pt]{article}

\usepackage{amsmath, amssymb,amsthm, mathrsfs}
\usepackage{graphicx,appendix}
\usepackage{xcolor}

\usepackage{hyperref,comment,bbm}
\usepackage{tikz,subcaption}

\hypersetup{
colorlinks=true, 
breaklinks=true, 
urlcolor= blue, 
linkcolor= blue, 
bookmarksopen=true, 
pdftitle={}, 
pdfauthor={Stephane Junca}, 
pdfsubject={} 
}

\usepackage{float}

\usepackage{import}
\usepackage{xifthen}
\usepackage{pdfpages}
\usepackage{transparent}

\newcommand{\R}{\mathbb{R}}
\newcommand{\Z}{\mathbb{Z}}

\newcommand{\tcb}{\textcolor{blue}}

\def\1{1\hspace{-.7ex} \rm{I}}

\newcommand{\bqr}{\begin{eqnarray}}
\newcommand{\eqr}{\end{eqnarray}}
\newcommand{\bqre}{\begin{eqnarray*}}
\newcommand{\eqre}{\end{eqnarray*}}

\newcommand{\alali}{$\mbox{ }$\\}

\newtheorem{theorem}{Theorem}[section]

\newtheorem{proposition}{Proposition}[section]
\newtheorem{lemma}{Lemma}[section]
\newtheorem{definition}{Definition}[section]
\newtheorem{remark}{Remark}[section]

 \numberwithin{equation}{section}

\title{ Entropy solutions in $BV^{s}$  for a class of triangular systems involving a transport equation}

\author{
Christian Bourdarias \thanks{Universit\'e  Savoie-Mont Blanc, LAMA,  Chamb\'ery, France, bourdarias@univ-savoie.fr},
Anupam Pal Choudhury \thanks{School of Mathematical Sciences, 
National Institute of Science Education and Research, Bhubaneswar 752050, India, and 
Homi Bhabha National Institute, Training School Complex, Anushaktinagar, Mumbai 400094, India, anupampcmath@gmail.com},\\
 Billel Guelmame \thanks{Universit\'{e} C\^ote d'Azur, Inria \& CNRS, LJAD,   Nice,  France,  Billel.Guelmame@univ-cotedazur.fr}, 
 St\'ephane Junca \thanks{Universit\'{e} C\^ote d'Azur, Inria \& CNRS, LJAD,   Nice,  France,  Stephane.Junca@univ-cotedazur.fr}
}
\date{}

\begin{document}
\everymath{\displaystyle}
\noindent

\maketitle 

{ABSTRACT. --}{
 \   In this article, we consider a class of strictly  hyperbolic triangular systems involving a transport equation. Such systems are known to create measure solutions for the initial value problem. 
  Adding  a stronger transversality assumption on the fields, we are able to obtain solutions in $L^\infty$ under optimal  fractional $BV$ regularity of the initial data.
Our results show that the critical fractional regularity is $s=1/3$. We also construct an initial data that is not in  $BV^{1/3}$ but for which a blow-up in $L^\infty $ occurs, proving the optimality of our results.
}

\medskip


\medskip

 
\tableofcontents
 
\section{Introduction}  \label{sec:in}
We consider triangular systems of the form, $t>0$, $x \in\R$, 
\begin{align}
&\partial_t  u     +   \partial_x  f(u) = 0  ,  \label{eq:scl}\\
&\partial_t  v     +   \partial_x  \left( a(u) v \right)  = 0  \label{eq:lin}.
\end{align}
Here $f$ is the scalar flux function (which we shall henceforth refer to as the flux)  for the equation \eqref{eq:scl} and the function $a$ denotes the velocity of the linear  equation with respect to $v$ \eqref{eq:lin}.
The system has   a decoupled nonlinear conservation law  and a coupled ``linear'' transport equation  with a discontinuous velocity. 
The above system is complemented by a set of initial data, $x \in \R$, 
\begin{align}\label{u0}
u(x,0)=u_0(x), \\ 
 v(x,0)=v_0(x).\label{v0}
\end{align}
 The Pressure Swing Adsorption  process  (PSA)  in chemistry  \cite{BGJ0}  has such  a triangular structure  after a change of variables from Euler to Lagrange \cite{BGJP,P07}.
Such systems are already of mathematical interest due to the coupling  of the theory of scalar conservation laws with the theory of transport equations.  This system was  studied in  \cite{LF90} in a non hyperbolic setting, $f'=a$,  with measure solutions for $v$. 
That even a strictly hyperbolic setting is not enough to avoid measure solutions was shown in \cite{DM08,HL}.
 Here, strengthening the hyperbolicity of the system by a stronger transversality condition,  global weak  bounded entropy solutions are provided with an optimal fractional $BV$ regularity for the initial data $u_0$. 

  
At the first sight, such  a system seems easy to solve in the ``triangular'' manner, that is solving the first equation \eqref{eq:scl} to get $u$ and then solving the linear equation \eqref{eq:lin} keeping $u$ fixed.
This method works well for smooth solutions \cite{GGJJ}, but, when a shock wave appears in $u$ the velocity $a(u)$ becomes discontinuous.   The theory of linear transport equations with discontinuous velocity is a delicate topic  yielding measure solutions and a loss of uniqueness \cite{BJ98,BJ99,PR}. 

 In this paper, we propose a different approach to obtain  global weak solutions in $L^\infty$. 
A main idea is to consider the system  \eqref{eq:scl}, \eqref{eq:lin} not as a triangular system, but,  as a $2\times2$ hyperbolic system as in \cite{BaJe}.    If $f$ is nonlinear, one field is nonlinear and the other one is linearly degenerate.  When $f$ is uniformly convex it is a particular case of $2\times2$ systems with one genuinely nonlinear field and a linearly degenerate one \cite{HJ}.
 An example of such a system when the flux  $f$ is piecewise convex or concave arises from gas-liquid chromatography \cite{BGJP}.    
 The presence of linearly degenerate fields can simplify the study of solutions  \cite{P98,P07},
but can also produce blow-up behaviours \cite{BGJP,NS}.  
 
We show that the behaviour of  the proposed entropy solutions  is linked to the fractional $BV$ regularity of $u$, indeed, $u_0 \in BV^s$, $0<s \leq 1$.  
  The  $BV^s$ framework is optimal to study the  regularity of  entropy solutions of scalar conservation laws \cite{CJ1,EM}. This framework is recalled in Section \ref{sec:basic}. For the scalar case, the theory works well for all $s>0$. For the triangular system \eqref{eq:scl}, \eqref{eq:lin}, we prove  that the regularity $s=1/3$ of $u_0$ is critical  for the existence of $L^\infty$ entropy solutions.  
  The exponent $s=1/3$ is directly linked to the  cubic estimate on the Lax curves \cite{Lax57}.  For nonconvex fields, it is known that the Lax curves are less regular \cite{B02,AM04,LF02}. However, for the triangular system \eqref{eq:scl}, \eqref{eq:lin} with the flux $f\in C^4$,  the velocity $a \in C^3$, and satisfying a uniformly strict hyperbolicity assumption, we prove that the exponent $s=1/3$ is optimal for the existence theory  of  bounded entropy solutions. 
  Our proof is based on a new cubic estimate that holds even when the flux $f$ is non-convex.
 Our estimate generalizes the well known cubic estimates for genuinely nonlinear (convex) systems  by  Lax \cite{Lax57}.

The paper is organized as follows. The hyperbolicity of the triangular system  and the key assumptions are given in Section \ref{sec:triangle}.
 The definitions of weak and  entropy solutions  are stated in Section \ref{ss:entropy}.
 The $BV^s$ framework is recalled in Section \ref{sec:basic}.
 In Section \ref{sec:main}, the two main results are stated: existence  for  {\color{black}{$s \geq 1/3$}} and blow-up  for $s < 1/3$.
In Section \ref{sec:hyp}, the  Riemann invariants and  the Riemann problem are studied. 
The Lax curves  and the key cubic estimates are studied in Section \ref{s:Lax-curves}.  The proof of existence is discussed in Section \ref{sec:exist} and the blow-up is studied in Section \ref{sec:boum}. In the appendices \ref{lipschitz} and \ref{uniq-riem}, we discuss on the uniqueness of the Riemann problem in the class of bounded weak entropy solutions.

\subsection{The hyperbolic triangular system} \label{sec:triangle}

The system  \eqref{eq:scl}, \eqref{eq:lin} of conservation laws  is  hyperbolic  when $f'\neq a$ and can be rewritten, using the vectorial flux $\textbf{F}$, as $\partial_t  \textbf{U}+ \partial_x \textbf{F}(\textbf{U})=0 $, $\textbf{U}=(u,v)^\top$. The  matrix of the linearized system has a triangular structure,
$$
 D\textbf{F}(\textbf{U})=  \left(  \begin{array}{cc}  f'(u) & 0  \\ a'(u) v &  a(u) \end{array}\right).
$$ 
This sytem  has the unbounded invariant region $[-M,M]_u \times \R_v $,
 where $M = \|u_0\|_\infty$.\\
 In this paper, the system is assumed to be strictly hyperbolic through the condition
\begin{align}\label{Hyp:str-hyperbo}
  \forall u  \in  [-M,M],  \qquad   f'(u)   >  a(u).
\end{align}
Of course, the symmetric asumption: $\forall u  \in  [-M,M],\,   f'(u) < a(u) $, yields a similar study. \\ 
For large data, the strict hyperbolicity condition \eqref{Hyp:str-hyperbo} has to be strengthened on the set $[-M,M]$, by the following uniformly strict hyperbolicity (USH) condition:
 \begin{align}\label{USH}
  \inf_{|u| \leq M} f'(u)   >   \sup_{|u| \leq M} a(u).
\end{align}
An interesting  case  is already when $f$ is convex, $f''>0 $ everywhere, or concave  $f''<0 $.   This case  occurs for  a chromatography system with  a convex isotherm written in appropriate Lagrangian coordinates \cite{BGJP}.   
 \\
 \noindent
In this paper,  the flux $f(\cdot)$ belongs  to   $C^4$ and the velocity $a(\cdot)$ belongs to  $C^3$.
Few times  in the paper,  to present the nonconvex case, the flux is assumed to have a finite number of inflection points,
\begin{align}\label{Hyp:loc-finite}
  \mathcal{Z}= \{u,   \; f''(u) = 0,\; |u| \leq M  \}  \mbox{ is finite}.
\end{align}
Thus, Riemann problems  are solved by  a finite number of elementary waves, the so called composite wave \cite{HR}. The finiteness of the number of inflection points  is not mandatory for the existence Theorem \ref{main} below, because  the proof uses  approximate  piecewise linear  flux $f_\nu$.  The important point is that $f_\nu$ is also 
 locally piecewise convex or concave so the solution of the Riemann problem has only a finite number of wavefronts \cite{D72}.

\subsection{Weak and entropy solutions}  \label{ss:entropy}
A weak solution of  the system \eqref{eq:scl}-\eqref{eq:lin} satisfying the initial conditions  \eqref{u0}-\eqref{v0} is defined as follows.
\begin{definition}[weak solution] \label{weak-solution}
The pair $(u,v)$  is a weak solution of the system  \eqref{eq:scl}-\eqref{eq:lin} 
with initial data \eqref{u0}-\eqref{v0} if  for all compactly supported  test functions $\varphi, \psi \in C^1_c(\R \times [0,+\infty[, \R)$, the following integral identities hold:
\begin{align}
\int_0^{+\infty} \int_\R    \left( u \, \partial_t \varphi +  f(u) \, \partial_x \varphi  \right)  dx \, dt 
+   \int_\R    u_0(x)   \,\varphi(x,0)  dx     & = 0, \label{int-identity-1}  \\
\int_0^{+\infty} \int_\R    \left( v  \, \partial_t \psi +  a(u)\, v  \,\partial_x \psi \right)  dx \, dt 
+   \int_\R    v_0(x)  \, \psi(x,0)  dx     & = 0. \label{int-identity-2}
\end{align}
\end{definition} 
The  following  regularity is usually required for $(u,v)$: $u \in L^\infty_{loc}$,  $v \in L^1_{loc}$ or  $v$ is  a measure.  
In the case when $v$ is a measure, there are some issues in defining the product $a(u) \times v$ \cite{BJ98}. 
 But our main focus in this paper is on bounded weak solutions $u, v \in L^\infty$.

We propose below a notion of entropy solutions for the system \eqref{eq:scl}-\eqref{eq:lin}. 
As in \cite{BGJ2,JuLo5} the entropy condition is only tested on the nonlinear component $u$.  Contact discontinuity waves linked to a linearly degenerate field are well known not to affect the entropy inequality which, therefore, remains an equality \cite{D16}.
\begin{definition}[entropy solution]
The pair $(u,v)$ is an entropy solution of the system  \eqref{eq:scl}-\eqref{eq:lin} 
with initial data \eqref{u0}-\eqref{v0} if it is a weak solution  and for all convex function $\eta$ and all non-negative test functions $\varphi\in C^1_c(\R \times ]0,+\infty[, \R)$,  with $q'=\eta'  f'$, $u$ satisfies the following inequality:
\begin{align}
\int_0^{+\infty} \int_\R    \left( \eta(u) \, \partial_t \varphi +  q(u) \, \partial_x \varphi  \right)  dx \, dt 
+   \int_\R    \eta(u_0(x))  \, \varphi(x,0)  dx  
 \geq 0 . 
\end{align}
\end{definition} 

Thus,
it suffices only to have a weak solution of the system and the entropy solution of 
\eqref{eq:scl}.  

Uniqueness of entropy solutions with a fixed initial data $(u_0,v_0)$ is a delicate matter \cite{Br00}. It is due to the lack of uniqueness of weak solutions for the linear transport equation with a discontinuous velocity (\cite{BJ98, BJ99}). \\
%
%


%
For triangular systems as discussed in this paper, one can have measure solutions.
A prototypical example of such triangular systems  \cite{HL} is given by 
\begin{equation}
 \begin{aligned}
 \partial_t  u     +   \partial_x  \left[\frac{u^{2}}{2}\right] &= 0  ,  \\
\partial_t  v     +   \partial_x  \left[ (u-1) v \right]  &= 0,
 \end{aligned}
 \label{example}
 \end{equation}
 with $f(u)=\frac{u^2}{2} $ and $a(u)=u-1 $. Clearly this satisfies the strict hyperbolicity condition \eqref{Hyp:str-hyperbo}  but not the more restrictive uniform strict hyperbolicity condition \eqref{USH} for too large data. 
 %
%
Similar phenomenon is well known for transport equations  \cite{BJ98} and is also widely observed in the case of non-strictly hyperbolic systems \cite{BJ99}.

The main approach in solving the system \eqref{example} in \cite{HL} was to observe that the first equation in $u$ can be solved independently. The second equation can then be considered as a transport equation in $v$ with a discontinuous coefficient $u$ (see also \cite{BJ98}),
but this idea of looking at the two equations separately might not be a good one. Our aim in this article is to prove the existence of solutions to systems of type \eqref{example} in the class of fractional BV functions and hence one need not appeal to $\delta$-shock wave type solutions.  
Moreover, to obtain $L^\infty$ solution with a wave front algorithm, we suppose that the system \eqref{eq:scl}, \eqref{eq:lin} is uniformly strictly hyperbolic \eqref{USH}. Note that we already need the uniform strict hyperbolicity condition \eqref{USH} to solve the Riemann problem appropriately. 
\subsection{$BV^s$ functions} \label{sec:basic}
 For  one dimensional scalar conservation laws, the  spaces $BV^s$  are known to give optimal results   for weak entropy solutions  \cite{BGJ6,CJ1,EM,EMproc}, 
 first on the fractional Sobolev regularity, 
   and  second on the structure of such solutions.
    Such results on the  maximal regularity have been extended for some systems  and the multidimensional case \cite{CJ2, GGJJ}.
\\
In this section, basic facts on $BV^s$ functions are recalled.

\begin{definition}\cite{MO}
A function $u$ is said to be in $BV^{s}(\mathbb{R})$ with $0<s\leq 1 $ if $TV^{s} u < +\infty $, where 
\begin{equation}
TV^{s}u:= \sup_{n \in \mathbb{N},\ x_1<\dots<x_n} \sum_{i=1}^{n} \vert u(x_{i+1})-u(x_{i}) \vert^{1/s} .
\notag
\end{equation}
The $BV^{s}$ semi-norm is defined by
\begin{equation}
\vert u \vert_{BV^{s}} := (TV^{s}u)^{s},
\notag
\end{equation}
and a norm on this space is defined by
\begin{equation}
\| u \|_{BV^s}:= \| u \|_{L^{\infty}}+| u |_{BV^{s}}.
\notag
\end{equation}
\end{definition}
 Fractional $BV$  functions have  traces like $BV$ functions.   This is a fundamental property  to define  the Rankine-Hugoniot condition for shock waves. Morerover, this property is not true for the Sobolev functions in $W^{s,1/s}$, the  Sobolev space nearest to $BV^s$ \cite{BGJ6}.
\begin{theorem}\cite{MO}
 For all $s \in ]0,1[$, $BV^s$ functions are   regulated functions.
\end{theorem}
 The   fractional total variation only depends on the local extrema of the function and the order of this extrema. 
\begin{lemma}\cite{BGJ6,HJ}
 If $u$  is a piecewise monotonic function and if its  local extrema  are located in the increasing sequence  $(x_i)_{i \in I}$, then  $TV^s u$ only depends on the sequence $ (u(x_i))_{i \in I}$. 
Moreover, there exists  an ordered subset $J$ of $I$ such that   
$$
TV^s u =    \sum_{j \in   J,  j \neq \max J } \vert u(x_{suc(j)})-u(x_{j}) \vert^{1/s},
$$ 
where $suc(j)$ denotes the successor of $j$ in $J$, $suc(j)=\min\{ k \in J, k>j\}$.
\end{lemma}
Moreorever, it can be  dangerous to refine the mesh to compute the  fractional total variation \cite[example 2.1]{BGJ6}, \cite{Bruneau74}.     
Consider $u(x) \equiv x$ on $[0,R]$,  $p=1/s >1$. Then  $ TV^s u [0,R]= R^p$  but, 
when $n \rightarrow + \infty$, 
$$     
   \sum_{i=1}^{n} \vert u( i  R/n)-u((i-1)R/n) \vert^{p} = n (R/n)^p= R^p/ n^{p-1}
\rightarrow 0.
$$
 This property, which is not true in $BV$,  is used later to prove the existence  of weak solutions in $BV^{1/3}$ for the triangular system \eqref{eq:scl}, \eqref{eq:lin} when the flux $f$ is convex. 
\section{Main results} \label{sec:main}

The main results are stated  in the $BV^s$ framework. The basic facts on this setting were recalled in  Section \ref{sec:basic}.  
{\color{black}{The critical space for the existence theory is $BV^{1/3}$.}}
%
\begin{theorem}[Existence in $L^\infty$  with $ (u_0,v_0) \in BV^{1/3} \times L^\infty$]\label{main}
Suppose that the flux $f(\cdot)$ and the  transport velocity $a(\cdot)$ satisfy the following assumptions on 
$[-M,M]$ where  $M= \|u_0\|_\infty$:
\begin{enumerate}
\item   $f'(\cdot)$ and $a(\cdot)$  belong to $C^3([-M,M],\R)$,
\item  $f' (u)> a(u)$ and  satisfy the uniform strict hyperbolicity assumption (USH) \eqref{USH}, 
\end{enumerate}
 Then, if $(u_0,v_0)$ belongs to  $BV^{s} \times L^\infty$   and $s \geq 1/3$, there exists an entropy  solution  $(u,v)$ of the system \eqref{eq:scl}, \eqref{eq:lin},  
$  u \in  L^\infty ( [0,+\infty)_t, BV^{s}(\R_x,\R) ) , 
\quad  v \in  L^\infty ((0,+\infty)_t \times \R_x, \R). $ \\
%
In addition, the positivity of the initial data is preserved, that is, if $ \inf v_0 > 0$ then $ \inf v > 0$.
\end{theorem}
%
 Two mandatory conditions are required to avoid singular phenomena at $t=0+$.
The strong transversality condition (USH) \eqref{USH} is the key assumption to  avoid instantaneous $\delta$-shock.  The $BV^s$ regularity of $u_0$ is the other key condition to avoid instantaneous blow-up in $L^\infty$.

{The positivity of $v$  is the key tool to get the uniform strict hyperbolicity of the (PSA) system \cite{BGJ2}. The complete proof of the positivity of $v$ for the (PSA) system  is obtained in \cite{BGJ3} without using the triangular structure of the (PSA) system \cite{BGJP}.}

Notice that the sign of $v$ is preserved a.e. by Theorem \ref{main}.  Of course, by linearity of the equation \eqref{eq:lin} with respect to $v$,  the negativity of the initial data is also preserved, that is, if $\sup v_0 < 0$ then $\sup v < 0$.
\\

The triangular approach, that means solving first  \eqref{eq:scl}  with the unique Krushkov entropy solutions \cite{K} and then  \eqref{eq:lin},  usually yields non-unique  measure solutions $v$ \cite{BJ98,PR}. 
It is the reason why the uniqueness of entropy solutions for the whole triangular system is an open problem.  Of course, we have  uniqueness for $u$ but  the problem of uniqueness remains for $v$. For a weakly coupled system with some linearly degenerate fields, the entropy condition only on the nonlinear fields is enough to ensure uniqueness in \cite{JuLo5}. For the triangular system, the coupling by the transport velocity is  too nonlinear to achieve uniqueness in the same way.
\\

For the existence result,  a  Wave Front Tracking algorithm (WFT)  is proposed  as in \cite{BaJe} for the whole system.
As a matter of fact the  approximate solution $u$  of the scalar equation satisfies the $BV^s$  uniform bounds  as in \cite{BGJ6,JR2}.
The  difficult point is to bound $v$ in $L^\infty$.  Since the system is linear with respect to $v$, an $L^\infty$ bound for $v$ is enough to get an existence theorem as in \cite{BGJP}.

Theorem \eqref{main} is optimal and, in general,  we cannot reduce the $BV^{1/3}$ regularity of $u_0$, else a blow-up can occur. For this purpose, we build a sharp  example in  a space very near to $BV^{1/3}$, namely, $ BV^{1/3-0} $
\begin{align*}
BV^{1/3} \subsetneq \quad BV^{1/3-0} := \bigcap_{s <1/3} BV^s   .  
\end{align*}  
 The  simple example  \eqref{example} from \cite{HL} with a Burgers' flux $f$ and a linear velocity $a$ is enough to provide a blow-up. 
\begin{theorem}[Blow-up in $L^\infty$ at $t=0+$ for $u_0 \in BV^{1/3-0}$]\label{blow}
For the system \eqref{example} with the  Burgers' flux $ f(u)=u^2/2$ and the velocity $a(u)=u-1$, there exists $u_0 \in BV^{1/3-0}$, $v_0 \in  L^\infty$ such that 
\begin{itemize}
\item the system \eqref{eq:scl}, \eqref{eq:lin} is uniformly strictly hyperbolic (USH) \eqref{USH}, and
\item there doesn't exist any bounded entropy solution of the system \eqref{eq:scl}, \eqref{eq:lin} with the initial data $u_{0},v_{0} $.
\end{itemize}
\end{theorem}
%
Of course $u_0 \notin BV^{1/3}$, otherwise the existence theorem \ref{main}  gives a bounded entropy solution. 
There is no blow-up for $u$ since  the entropy solution $u$  of the scalar conservation law  satisfies the maximum principle. Indeed, only the function $v$  has a blow-up  at time $t=0^+$.
 Notice  also that the blow-up depends only on the regularity of $u_0$ and not of $v_0$.
 
 The example presented below follows the construction of particular explicit solutions \cite{AGV1}
 which provide optimal example in $BV^s$ for convex scalar conservation laws \cite{GGJJ}.  On any compact set avoiding the blow-up  point, the entropy solution $u$ is locally Lipschitz except on a finite number of lines.  For initial data $u_{0}$ leading to such an entropy solution $u$, we show that a uniqueness result for the  entropy solution $(u,v)$ holds at least upto the time of first interaction of waves in $u$.  Thus there is no way to avoid the blow-up. In general, a possible dense jump set in plane $(t,x)$ is possible for the entropy solution $u$ of a scalar conservation law \cite{densejump}.


Another consequence of this blow-up example is that for general nonlinear $2\times2$ strictly hyperbolic systems with a genuinely nonlinear eigenvalue and a linearly degenerate one, the existence result proven in \cite{HJ} is optimal.

\section{The  uniformly strictly hyperbolic system} \label{sec:hyp}
In this section, we study the Riemann invariants and the solution of the Riemann problem for the triangular system \eqref{eq:scl}-\eqref{eq:lin}.   

  \subsection{Riemann invariants}
  
  A $2\times 2$ strictly hyperbolic system admits, at least locally, a set of coordinates which diagonalizes the hyperbolic system for smooth solutions \cite{Sm}.
 The knowledge of this coordinate system, given by the Riemann invariants, is often useful in understanding the structure of the system. Next we study the Riemann invariants for the system \eqref{eq:scl}-\eqref{eq:lin}.
  
 The eigenvalues of the system   \eqref{eq:scl}-\eqref{eq:lin} are 
  \begin{align}\label{eigenvalues}
  \lambda_1 = f'(u)     >    a(u)= \lambda_2. 
  \end{align}
  Notice that the eigenvalues are functions of $u$ only. Let $r_{1}$ and $r_{2}$ denote the corresponding right eigenvectors.  
  %
%

Clearly, $u$ is a 2-Riemann invariant associated to the right eigenvector  $r_2= (0,1)^\top$
  and satisfies
\begin{equation}
\partial_{t} u+f'(u) \partial_{x} u=0.
\notag
\end{equation}
A 1-Riemann invariant, which we denote as z(u,v), corresponding to the right eigenvector $ r_1$, can be computed in the following manner.
We note that a right eigenvector of the matrix $D\textbf{F}(u)$ corresponding to the eigenvalue $f'(u)$ is given by 
\begin{equation}
r_1 = \begin{pmatrix}
1 \\
\frac{a'(u) v}{f'(u)-a(u)}
\end{pmatrix}.
\notag
\end{equation} 
Then $z$ satisfies 
\begin{equation}
\frac{1}{v}  \partial_{u} z
        =  \frac{a'(u)}{a(u)- f'(u) }   \partial_{v} z.
\notag
\end{equation}
This can be solved using a separation of variables.   
\begin{align*}
         \partial_{u} z=   \frac{a'(u)}{a(u)- f'(u) } ,  &  &
          \partial_{v} z=   \frac{1}{v} .
  \end{align*}
 For instance, a Riemann invariant is given by
  \begin{align*} 
  z(u,v)=   A(u) + \ln(|v|), &   &
  A'(u) =     \frac{a'(u)}{a(u) - f'(u)}.
  \end{align*}
  To avoid the singularity at $v=0$, it suffices to take the exponential    
  \begin{align} Z= \exp(z)= v \exp (A(u)),
  \end{align} 
  which satisfies the transport equation
  \begin{equation}
  \partial_{t} Z+a(u) \partial_{x} Z=0.
  \label{Z-transport}
  \end{equation}

\subsection{The Riemann problem}     \label{ss:Riemann-pb}
We study  the Riemann problem for the system \eqref{eq:scl}-\eqref{eq:lin}  with initial data:
\begin{align}
  u_0(x)=      u_\pm ,  &   \pm x > 0, \\
    v_0(x)=      v_\pm ,  &   \pm x > 0.
\end{align}

A direct and somewhat naive approach is to solve the conservation law \eqref{eq:scl} first and then  the second equation \eqref{eq:lin}  using the solution $u$.  In such an approach, one faces the difficulty  of solving the linear transport equation with a discontinuous coefficient.

Instead, we consider the two equations together as a system. This is a key point as in \cite{BaJe}.
The solution of the Riemann problem consists of two waves separated by an intermediary state 
$(u_m,v_m)$ where $u_m=u_-$ and $v_m$ is unknown. 
 In accordance with the labeling of the eigenvalues \eqref{eigenvalues}, a wave associated to $u$ is called a 1-wave and a wave associated to $Z$ is called a 2-wave. 
\begin{enumerate}
\item  A wave associated to the eigenvalue  $\lambda_1=f'$ is a shock  wave or a rarefaction wave if $f$ is convex or concave.   For a non-convex flux $f$, this wave is a composite wave.
\item       A linearly degenerate wave associated to the eigenvalue  $\lambda_2=a$ is called a contact discontinuity. The speed of this 2-wave is $a(u_-)$.
\end {enumerate}
 The intermediate value $v_m$ has to be computed through the 1-wave.   
Now, the various 1-waves that can occur are detailed.  For this purpose, we consider Riemann problems yielding only a 1-wave.

  \subsubsection*{Shock waves}
  Let us denote $U=(u,v)$ and $\tilde{U}=(u,Z) $. \\   
%
  The Rankine-Hugoniot condition gives 
  \begin{align} \label{RH}
     s[u]=  [f(u)] ,     \quad    s[v]= [a(u)\, v],
  \end{align}
  where $s$ denotes the speed of the discontinuity (or the slope of the jump).
  
 Thus, the slope of the jump is determined by $u_\pm$ and the flux $f$,
  \begin{align}\label{shock-slope}
    s =  \frac{[f(u)]}{[u]}.
  \end{align}
   Since entropy solutions are considered, the Oleinik-entropy condition \cite{Br00} for \eqref{eq:scl} enforces  the Lax-entropy conditions,
    \begin{align}\label{Oleinik}
      f'(u_-) \geq  s \geq f'(u_+),
    \end{align}
    whence  \begin{align} s \geq f'(u_+) >  a(u_+). \end{align}\\
    Using the second equation of \eqref{RH} yields  
    \begin{align}   \label{eq:Lax-shock-curve}
     (s- a(u_+)) v_+ =  (s - a(u_-)) v_- ,\\
    \label{eq:Lax-shock-curve-} 
         v_+ =         v_-  \frac{ s    - a(u_-)}{ s - a(u_+)}:= \mathcal{S}_-(u_+;U_-).
    \end{align}
   The equation \eqref{eq:Lax-shock-curve-} can be interpreted in terms of the Lax shock curve. For a fixed $U_{-}=(u_-,v_-)$, the right hand side of  \eqref{eq:Lax-shock-curve-}  is only a function of $u_+$ as $s$ is given by \eqref{shock-slope} as  a function of $u_+$.  In the plane $U=(u,v)$,  \eqref{eq:Lax-shock-curve-} describes the set of $U_+$   such that  the Riemann problem with initial data $U_\pm$  is solved by a shock.   On this curve, only $U_+$ satisfying the Oleinik condition  \eqref{Oleinik} are considered to allow an entropic shock. 
  
   Conversely, if $U_+$ is fixed, the Lax shock curve is parametrized by $u_-$ and reads
    \begin{align}\label{eq:Lax-shock-curve+} 
         v_- =       v_+    \frac{ s    - a(u_+)}{ s - a(u_-)}:= \mathcal{S}_+(u_-;U_+).
    \end{align}

Proceeding similarly with  $Z=v\ \exp(A(u))$, $\tilde{U}_{-}=(u_-,Z_-)$
 and keeping the notation $\mathcal{S}$ in coordinates $(u,Z)$ yields,
\begin{equation}
Z_{+}=Z_{-} \frac{s-a(u_{-})}{s-a(u_{+})} \exp \left[A(u_{+})-A(u_{-}) \right]:= \mathcal{S}_{-}(u_+;\tilde{U}_{-}),
\notag
\end{equation}
and
\begin{equation}
Z_{-}=Z_{+} \frac{s-a(u_{+})}{s-a(u_{-})} \exp \left[A(u_{-})-A(u_{+}) \right]:= \mathcal{S}_{+}(u_-;\tilde{U}_{+}).
\notag
\end{equation}

    \subsubsection*{1-Contact discontinuity}
  
    This is a limiting case of the preceding one, when $f'$ is constant on the interval $[u_-,u_+]$.   The same formula for the shock wave follows.
    \begin{align}
     s= f'(u_\pm)=  f'(u) >  a(u_\pm), \quad   u \in [u_-,u_+] , \\
      v_+ =   v_{-}        \frac{ f'(u_-)    - a(u_-)}{ f'(u_+) - a(u_+)}= \mathcal{S}_-(u_+;U_-),\\
      Z_{+}= Z_{-} \frac{f'(u_{-})-a(u_{-})}{f'(u_{+})-a(u_{+})} \exp \left[A(u_{+})-A(u_{-}) \right]:= \mathcal{S}_{-}(u_+;\tilde{U}_{-}).
    \end{align}  
    
    Such waves arise in the wave-front tracking (WFT) method when the flux is approximated by piecewise linear functions. \\
  
 \subsubsection*{Rarefaction waves}
  Since $Z$ is a 1-Riemann invariant, the Lax rarefaction curve $\mathcal{R}_{-}(\tilde{U}_{-}) $ is simply, 
  \begin{align}
   Z_+ =  Z_-.
  \end{align}
  Using the defiinition of $Z$, this implies $v_+ \exp(A(u_+))= v_- \exp(A(u_-)) $ and thus the  Lax rarefaction curve can be written explicitly,
  \begin{align}
   v_+ =  \mathcal{R}_-(u_+;U_-) = v_- \exp(  A(u_-)- A(u_+)), \\
      v_- =  \mathcal{R}_+(u_-;U_+) = v_+ \exp(  A(u_+)- A(u_-)).
  \end{align}
  Notice that $v_+$ and $v_-$  have the same signs. In particular, if $v_-=0$, $v=0$ is constant through the rarefaction wave.\\

   \subsubsection*{1-Composite waves}
   The composite waves only occur for the waves associated to the first eigenvalue $f'(u)$. 
  In general,  the entropy solution $u$ is  a  composite wave \cite{HR}. If $f$ has a finite number $N$ of inflection points, then there are at most $N$ contact-shock waves \cite{D16,LF02}. The Lax curves associated to such waves are studied below in Section \ref{s:Lax-curves}.

   \subsubsection*{2-Contact discontinuity}
     $u$ is a 2-Riemann invariant and hence is constant along the 2-contact discontinuity.  Thus  the Lax curve is simply a vertical line in the plane $U=(u,v)$ or the plane $\tilde{U}=(u,Z) $,
     \begin{align}\label{2-Lax-curve}
          u_-=u_+   .       
     \end{align}

\section{The Lax curves}  \label{s:Lax-curves}

A fundamental theorem due to Lax \cite{Lax57} states that the shock curve and the rarefaction curve emanating from a constant state $U_-$  in the plane $(u,Z)$ or $(u,v)$ have a contact of the second order for a genuinely nonlinear eigenvalue.  This means that the shock curve can be replaced by the rarefaction  curve up to an error of  order $[u]^3$  where $[u]=u_+-u_-$ \cite{Br00,D16,Serre}.   
For the triangular system, a genuinely nonlinear eigenvalue  means  $f'' > 0 $   (or  $f'' <0 $) everywhere.  For nonconvex cases, typically $f''$ locally has a finite number of roots where $f''$ changes its sign and the Lax curves are less regular due to the occurrence of contact-shocks.  
Under a concave-convex assymption, which means here that $f'''$ does not vanish, the regularity of the Lax curves is only piecewise $C^2$ \cite{LF02}, see also \cite{AM04}. As a consequence, the error becomes of order $[u]^2$ for the variation of $Z$ through a contact-shock \cite{LF02}. The situation is worse in general, the Lax curves are only Lipschitz \cite{B02}.
  However, we prove that  cubic estimates are still valid for the triangular system \eqref{eq:scl}, \eqref{eq:lin}. It is  mainly due to the existence of Riemann invariant coordinates.

In this section, cubic estimates on the Lax curves for the triangular system in the plane  $(u,Z)$ are proved.

Let  $(u_-,v_-)$, $(u_+,v_+)$ be two constant states connected by a rarefaction or a shock wave. 
It is more convenient  to use the Riemann invariant coordinates $(u,Z)$.  
\\
The rarefaction curves in the plane $(u,Z)$ are simply,
\begin{align}
     Z_-=Z_+ ,
\end{align}
which means that in the  $(u,v)$ plane 
\begin{align}
     v_- \exp(A(u_-))=v_+ \exp(A(u_+)) .
\end{align}
For the shock curve, the Rankine-Hugoniot condition is written in the conservative variables $(u,v)$, 
 \begin{align}
    s  & = \frac{[f(u)]}{[u]} =   \frac{f(u_+)- f(u_-)}{u_{+}-u_{-}},   \\
   v_- ( s - a(u_-))   & =    v_+ (s - a(u_+)).
 \end{align}
 Since $f \in C^4$, $s$ is a $C^3$ function of its arguments.  Moreover, fixing  $(u_+,v_+)$ and considering $u_-$ as a variable, the Lax  shock curve is $C^3$ with respect to $u_-$, 
 \begin{align} 
   v_-  & =    v_+  \frac{s - a(u_+)}  { s - a(u_-)}.
 \end{align}
 Indeed,  the denominator never vanishes due to the uniformly strict hyperbolicity assumption (USH) \eqref{USH}.
 The same regularity of the shock curve holds  in the variables $(u,Z)$
  \begin{align} 
   Z_-  & =    Z_+  \frac{s - a(u_+)}  { s - a(u_-)} \exp( A(u_-)-A(u_+)).
 \end{align}
Of course, the Lax rarefaction curve and the Lax shock curve has to be restricted on the subset satisfying entropy conditions. Nevertheless, we  use these curves for all range of $u_-$ in $\R$  (at least for $-M \leq u_- \leq M$) to obtain a generalized Lax cubic estimate for the nonconvex case.

%
  \subsection{The Lax cubic estimate on the Rankine-Hugoniot curve} \label{ss:RH-cube}
   The Lax cubic estimate \cite{Lax57} can be written as follows for a shock  wave connecting 
    $(u_-,Z_-)$ to $(u_+,Z_+)$ for the triangular system \eqref{eq:scl},\eqref{eq:lin}, as soon as $Z$ is bounded,  
   \begin{align}
     [Z]= \mathcal{O}([u])^3,   && [Z]=Z_+ -Z_-, &&  [u]= u_+-u_-.
   \end{align}
  The Riemann invariant $Z$ is constant along the rarefaction curves. The Lax cubic estimate means that the shock curve and the rarefaction curve have a contact of order 2. 
  The Lax cubic estimate was written in a genuinely nonlinear framework \cite{Lax57}. This means that $f''$ does not vanish.  Indeed, the Lax computations are still valid without this convex assumption and without only considering the entropic part of the Rankine-Hugoniot curve. Here, we also use the non-entropic part of the Rankine-Hugoniot curve.  It is a useful tool  used many times in this paper, first in the next section  \ref{ss:cubicRp} to get a cubic estimate for  the  entropy solution of the Riemann problem for  the triangular system with a non-convex flux $f$.

   For the triangular system, the Rankine-Hugoniot curve is global, explicit and well defined thanks to the uniformly strict hyperbolicity assumption \eqref{USH}. For $2\times2$ systems, such a global curve does not always exist \cite{KK90}.
   
   Now, to prove cubic estimates, we have to write the Rankine-Hugoniot curve.
    Here, we choose  to  write  $Z_-$ as a  function of $u_-$  when $(u_+,Z_+)$ are fixed  for the following reasons.
  \begin{enumerate}
  \item  To solve the Riemann problem and compute the intermediary state $v_m$ which corresponds to $Z_m$ and here $Z_-$.
  \item  To obtain the cubic estimate on the global Rankine-Hugoniot curve  below.
  \item  To obtain the existence result, bounding $Z$ along the 2-characteristics.
  \item   To build a blow-up, again computing $Z$ from the right to the left on the 2-characteristics.
  \end{enumerate}
 The Rankine-Hugoniot curve   $\mathcal{RH}_{+}$ when $\tilde{U}_{+}=(u_+,Z_+)$ is fixed and $Z_-$ is a function of $u_-$, is given by
 \begin{equation}\label{def:RH}
Z_{-}=Z_{+} \frac{s-a(u_{+})}{s-a(u_{-})} \exp \left[A(u_{-})-A(u_{+}) \right]
    =     Z_+   r(u_-,u_+) := \mathcal{RH}_{+}(u_-; u_+, Z_+ ) . 
\end{equation}
Note that the Rankine-Hugoniot curve $\mathcal{RH}_{+} $ is a property of $(u,v)$, which means this curve is always computed in the conservative variables $(u,v)$ \cite{D16}. For our convenience, here we write it in terms of $(u,Z)$ due to the fundamental role of the line $Z=constant$ in the Lax cubic estimate \cite{Lax57}.
 This classic  cubic estimate on the shock curve is generalized on the global Rankine-Hugoniot curve. 
 \begin{lemma}[Cubic flatness of the global Rankine-Hugoniot curve]\label{lem:flatRH}
 If $f'$ and $a$ belong to $C^3(\R,\R)$ and satisfy the uniform strict hyperbolicity condition \eqref{USH}, then
 \begin{align}
  \notag
  s &= s(u_-,u_+)= \frac{[f]}{[u]}= \frac{f(u_+)-f(u_-)}{u_+-u_-}  \qquad \in C^3([-M,M]^2,\R) ,\\
  \notag 
  r&= r(u_-,u_+) =  \frac{s-a(u_{+})}{s-a(u_{-})} \exp \left[A(u_{-})-A(u_{+}) \right] \qquad \in C^3([-M,M]^2,\R),\\
   r&  =   1  + \mathcal{O}(1) \, [u]^3 >0, \quad \forall (u_-,u_+) \in [-M,M]^2,   \label{eq:flatr}\\
   Z_- &  
   =   Z_+  ( 1 +  \mathcal{O}(1)  [u]^3 ),
      \quad \forall (u_-,u_+,Z_+) \in [-M,M]^2 \times \R ,
      \label{eq:flatRH}
   \end{align}
where the constant $\mathcal{O}(1)$ depends only on the derivatives of $f'$ and $a$ on $[-M,M]$.
Moreover, $Z_-$ has the same sign as $Z_+$, more precisely, 
\begin{align*}
 Z_+=0  & \Rightarrow Z_-=0, \\  Z_+ \neq 0  &\Rightarrow Z_- Z_+> 0. 
 \end{align*} 
 \end{lemma}
  The cubic flatness of the Rankine-Hugoniot curve $\mathcal{RH}_+$ is expressed in \eqref{eq:flatRH}.   It does not depend on the convexity of the fields, although in the classical textbooks \cite{Br00,D16,HR,LF02,Serre,Sm} some nonlinearity assumptions on the fields are given. The reason of these nonlinear assumptions is usually to introduce the rarefaction curve and the shock curve. Here,  the global Rankine-Hugoniot curve defined for all $u_- \in [-M,M]$ is the main subject without looking at the entropic parts of this curve.  A careful reading of the classical proof of the cubic estimate in textbooks shows that it is a geometric property of the Rankine-Hugoniot curve itself, without referring at the entropic or nonlinearity  assumptions.
  This geometric property is a consequence of the symmetry of the Rankine-Hugoniot condition with respect to $U_-$ and $U_+$  as explained in \cite{Serre}. It is very important in this paper to prove the cubic estimate for the non-entropic part of the Rankine-Hugoniot  curve for two reasons.
  \begin{enumerate}
   \item  The cubic estimates on the Lax curve (which is not piecewise $C^3$ for nonconvex $f$ \cite{LF02}) uses the global Rankine-Hugoniot curve, as discussed in the next section \ref{ss:cubicRp}.
   \item   Rarefaction wave fans are replaced by weak non-entropic jumps in the wave front tracking algorithm \cite{Br00}. The error in the weak formulation is controlled by the cubic estimate to pass to the limit and get a weak solution of \eqref{eq:scl}, \eqref{eq:lin}.   
  \end{enumerate}
 The proof appears  as a  direct consequence of \eqref{eq:flatr}.  An elementary and self-contained proof using only Taylor's expansions is proposed.  A more tedious computation can give the more precise result 
 \begin{align}
    r(u_-,u_+) & =    1  +    E  [u]^3 +  \mathcal{O}([u]^3),
 \end{align}
 where $E$ depends in a quite complicated way on the derivatives of $f'$ and $a$ at $u=u_+$. 
 The computation of $E$ is quite intricate and not useful here, except in the last section \ref{sec:boum} on the blow-up where a direct  computation of $E$ at $u_+=0$ is given when $f$ is quadratic and $a$ is linear. 
 
 Now, Lemma \ref{lem:flatRH} is proven.
 \begin{proof}   
 The positivity on $r$ is a consequence of the assumption \eqref{USH}.  This positivity implies that $Z_-$ has the same sign as $Z_+$.

To obtain \eqref{eq:flatr}, we use Taylor expansion. Here $u_+$ is  fixed and $u_-$  is the variable near $u_+$. 
 The notations $a_-=a(u_-)$, $a_+=a(u_+)$ and so on    are used to shorten the expressions.
 \begin{align*}
  a_- &= a_+ -a'_+[u] +  (a''_+/2) [u]^2  +  \mathcal{O}([u]^3),\\
   s & = \frac{[f]}{[u]} = \frac{f_- - f_+}{-[u]}=   f'_+ - (f''_+/2) [u]  + ( f'''_+/6) [u]^2+ \mathcal{O}([u]^3).
 \end{align*}
 The hyperbolic quantity $h$ is used,
  \begin{align} \label{eq:h}h:=f'-a >0.
  \end{align} 
The sign of $h$ is the consequence of   the strict hyperbolicity assumption \eqref{USH}. 
We start with the fraction part of $r$  and $h_+=f'_+ -a_+ \neq  0$.
\begin{align}
 \frac{s-a_{+}}{s-a_{-}}   &= 
 \frac{   (f'_+ -  a_+) - (f''_+/2) [u]  + (f'''_+/6) [u]^2      }
 { (f'_+ -a_+) - ( f''_+/2 - a'_+)  [u]  +  (f'''_+/6 - a''_+/2) [u]^2  }  + \mathcal{O}([u]^3)  \\
 \label{eq:frac1}
  &= 1  - \frac{a'_+}{h_+} [u]  + \frac{  h_+ a''_+  - a'_+ f''_++ 2 a'^2_+ }{2h^2_+}    [u]^2    +   \mathcal{O}([u]^3) . 
\end{align}
 
 For the term $\exp(-[A])$ in $r$, the Taylor expansion of $A'_-$ is used at the first order.
 \begin{align*}
     - A'_- & =  \frac{a'_-}{h_-} 
     = \frac{a'_+ - a''_+[u] }{h_+ - h'_+[u] }+  \mathcal{O}([u]^2) \\
       &=   \frac{a'_+}{h_+}  + \frac{ a'_+ h'_+ - a''_+ h_+ }{h^2_+}[u]  
       + \mathcal{O}([u]^2). 
      \end{align*}
      Integrating  with respect to $u_-$ yields, 
      \begin{align*}
   A_- -A_+&   =   \frac{a'_+}{h_+}[u]  
      + \frac{ a'_+ h'_+- a''_+ h_+  }{2 h^2_+}[u]^2 
           + \mathcal{O}([u]^3) .
 \end{align*}
  Since  $\exp(x)= 1 + x+ x^2/2 +  \mathcal{O}(x)^3 $, it yields
 \begin{align} \label{eq:exp1}
  \exp \left[A_{-}-A_{+} \right] &=  1 
   + \frac{a'_+}{h_+}[u]
      + \frac{   a'_+ h'_+ - a''_+ h_+  + a'^2_+   }{2 h^2_+}[u]^2     + \mathcal{O}([u]^3) .
 \end{align}
 Now, multiplying     \eqref{eq:frac1}  and \eqref{eq:exp1} yields  $r= 1 + \mathcal{O}([u]^3)$.
 \end{proof}

\subsection{Cubic estimates for the Riemann problem} \label{ss:cubicRp}

Now,  the intermediary state $Z_m$ of a Riemann problem has to be estimated.  For this purpose,
 the variation of $Z$  along a composite wave is studied.  When the flux is convex,  Lax proved that the variation of $Z$ is  a cubic order of the variation of $u$ \cite{Lax57}.  For a non-convex flux, it is well known that the Lax curve is less regular, piecewise $C^2$ \cite{LF02} or only Lipschitz \cite{B02}.  However, we prove that for our triangular system we are able to keep a cubic order.  This is mainly due to the existence of Riemann coordinates for $2\times 2$ systems and the cubic estimate for  the  global Rankine-Hugoniot  locus, Lemma \ref{lem:flatRH}.
As a consequence, we can prove a similar estimate for the variation of $Z$ over a composite 1-wave.  This improves the well known square root estimate  for concave-convex  eigenvalues \cite{LF02}, which correspond  to cubic degeneracies for $f$.  That means that for the triangular system \eqref{eq:scl}, \eqref{eq:lin}, the estimate is as precise as for  the convex case \cite{Lax57}.
\begin{proposition}[Variation of $Z$ through a composite wave]\label{pro:cub}
\alali
Assume that the flux has $N_{infl}$ points. \\
Let the states $\tilde{U}_{i}=(u_{i},Z_{i}),\ i=1,\dots,m$, where $u_{0}<u_{1}<\dots<u_{m} \ (m\leq N_{infl}+1) $ comprise a composite 1-wave, and 
$Z_-=Z_0$ and $Z_+=Z_m$. Then the total variation of $Z$ through a 1-wave is, 
\begin{align}
 \| Z \|_\infty  \leq   \vert Z_{+} \vert   \exp\left( \mathcal{O}(  \vert u_{+}-u_{-} \vert^3) \right) ,
 \label{ ineq: Z-bound}
 \\
 TV Z  \leq   \mathcal{O}(1)  |Z_+|  \vert  u_{+}-u_{-} \vert^{3}    .
 \label{ineq:Z-TV}
\end{align} 
\end{proposition}
Later, as the consequence of the existence proof, the assumption of the finite number of inflection points can be removed and the same inequalities \eqref{ ineq: Z-bound}, \eqref{ineq:Z-TV} hold.
\begin{proof}
We note that $Z$ is a 1-Riemann invariant and therefore it remains constant over a rarefaction wave. Moreover, if the states $\tilde{U}_{i} $ and $\tilde{U}_{i+1} $ are joined by a jump, by Lemma \ref{lem:flatRH} and the classical inequality $1+x \leq \exp(x)$, we have the estimate 
\[   \vert Z_{i} \vert
  =  \vert Z_{i+1}  r( u_{i},u_{i+1} )\vert=  \vert Z_{i+1} \vert   \left(1 +  \mathcal{O}(  \vert u_{i}-u_{i+1} \vert^3) \right) 
\leq   \vert Z_{i+1} \vert   \exp\left( \mathcal{O}( \vert u_{i}-u_{i+1} \vert^3 ) \right)  .\]
Summing up and noting that $u_{0}, u_{1},\dots,u_{m} $ are ordered, the estimate \eqref{ ineq: Z-bound} follows,
\[   \max_i\vert Z_{i} \vert
\leq   \vert Z_{+} \vert   \exp\left( \mathcal{O}( \vert u_{+}-u_{-} \vert^3 ) \right) .\]
Now, the $BV$ bound for $Z$ through the composite wave is computed. 
\begin{align*}
  Z_{i}-Z_{i+1} =  Z_{i+1}  (r( u_{i},u_{i+1} ) - 1)=  Z_{i+1} \mathcal{O}(  \vert u_{i}-u_{i+1} \vert^3), 
\\
 TV  Z \leq    \max_i\vert Z_{i} \vert    \sum_i  \mathcal{O}( \vert u_{i}-u_{i+1} \vert^3 ) \leq 
 |Z_{+}| \mathcal{O}( \vert u_{+}-u_{-} \vert^3 ) \exp\left( \mathcal{O}(  \vert u_{+}-u_{-} \vert^3 ) \right)  .
\end{align*}
Since $u$ satisfies the maximum principle, we further have $\exp\left( \mathcal{O}( \vert u_{+}-u_{-} \vert^3) \right)= \mathcal{O}(1) $ which only depends on the $L^\infty$ bound of the initial data $u_{0}$. Hence the  inequality \eqref{ineq:Z-TV} follows.
\end{proof}

\section{Existence in $BV^{s}$}   \label{sec:exist} 
In this section, Theorem \ref{main} is proved using a simplified Wave Front Tracking (WFT) \cite{Br00,HR} algorithm for such a triangular system \eqref{eq:scl}, \eqref{eq:lin}.
The  $BV^s$ estimates  for $u$ are a consequence of such estimates for scalar conservation laws \cite{BGJ6, JR1, JR2}.   The $L^\infty$ bound for $v$ and the proof of existence of a weak solution for the triangular system are based on an approach using the $BV^{1/3}$ regularity for $u$. 

\subsection{The Wave Front Tracking algorithm}
The WFT depends on an integer parameter $\nu >0$.  The approximate solutions will be denoted by $u_\nu, v_\nu, Z_\nu$.   We shall mostly use the Riemann invariant coordinates $(u_\nu,Z_\nu)$ except when passing to the limit in the weak formulation.

 This algorithm is explained in many books \cite{Br00,D16,HR} on hyperbolic systems. 
Taking advantage of the structure of the triangular system  \eqref{eq:scl}, \eqref{eq:lin}, we will mix the WFT for the scalar case \cite{D72} and for systems \cite{Br00,HR}.
   The main principle is to work with piecewise constant approximations. 
   \medskip 
   
   As in the scalar case \cite{Br00,HR,JR1,JR2}, the values of $u_\nu$ are taken on a uniform grid parametrized by  the integer $\nu$, that is,  
    \begin{align} u_\nu  \in   \nu^{-1} \Z .
    \end{align} 
 On the other hand, $v_\nu$,  or equivalently $Z_\nu$,  is not required  to stay on the uniform grid, that is, 
    \begin{align} 
    v_\nu,  Z_\nu \in \R,
    \end{align}
    as we  solve the Riemann problem with the exact flux $f$ to compute $v_\nu$,  or equivalently $Z_\nu$.   In this way, we can use the cubic estimate on the global Rankine-Hugoniot curves  of Lemma \ref{lem:flatRH}.
      A similar approach for a convex case was used in \cite{HJ}.
        \bigskip
 
 The initial data are approximated as follows.
 Let $N_0$ + 1 be the number of constant states for the approximate initial data $u_{0,\nu}, v_{0,\nu}$  and the corresponding $Z_{0,\nu}$. Here $u_{0,\nu} $ takes values on the grid and $N_{0}$ is a function of $\nu $.\\
{\color{black}{To ensure that $(u_{0,\nu},v_{0,\nu})$ converges towards the initial data $(u_0,v_0) $, we need the condition that $N_{0} $ goes to infinity as $\nu $ tends to infinity.\\ }}
The approximate initial data can be chosen to satisfy the following uniform estimates with respect to $\nu$ \cite{JR1,JR2}:
\begin{align}
   \|u_{0,\nu} \|_\infty  \leq    \|u_{0} \|_\infty ,  \\
   TV^s u_{0,\nu} \leq   TV^s u_{0},         \\
   \|v_{0,\nu} \|_\infty  \leq    \|v_{0} \|_\infty.    
\end{align}
Moreover, $( u_{0,\nu}, v_{0,\nu}) $ converges pointwise almost everywhere to $( u_{0}, v_{0}) $ when $\nu \rightarrow + \infty$ and therefore, the previous inequalities become equalities at the limit $\nu \rightarrow + \infty$.
\bigskip

 The flux $f$ is replaced by a piecewise linear continuous flux  $f_{\nu} $ \cite{D72} that coincides with $f$ on the uniform grid \cite[Ch. 6]{Br00}, \cite[p. 70]{HR}, 
\begin{align}\label{f-nu}
  f_\nu   ( k/\nu)=   f(k/\nu), &    & \forall k \in \Z.
\end{align} 
At $t=0+$, $N_0$ Riemann problem are solved. The approximate solution $u_\nu$  is the weak entropy solution of
\begin{align} \label{eq:scl-nu}
\partial_t u_\nu + \partial_x f_\nu(u_\nu)= 0, &  & u_\nu(x,0)= u_{0,\nu} (x).
\end{align}
With strong compactness on $u_\nu$, the Kruzkov entropy solution is recovered \cite{Br00,HR}.
\medskip

For $  (u_\nu,v_\nu)$ the approximate vectorial flux $F_{\nu}(u,v)=\left(f_\nu(u), a(u)v \right)$ is used, as detailed below.      
The reason for this choice is twofold. First, to have a simpler wave front tracking for the scalar equation for $u$ with a non-convex flux following the seminal paper \cite{D72} (this case is also very well detailed in the textbook \cite{HR}).  This wave front tracking algorithm does not lead to any consistency error in the weak formulation \cite{D72}.  
Secondly, to use the new generalized cubic estimates for $Z_\nu$ which also hold for the solution $(u_\nu, v_\nu)$ of the system with the approximate flux $F_{\nu}$, as $f_{\nu}=f$ on the grid   $\{ (u,v) \in (\nu^{-1}\Z )\times \R \}$. 
\smallskip

After some time, some of the Riemann problems interact and hence, new Riemann problems have to be solved.  
The non-linear interactions for the triangular system are described in detail
below.
 The process continues until the second time of interactions and so on. 

This process can be continued for all time
because
 there is only a finite number of interactions for all time.  
 For $u_\nu$ this follows from the existing results for the scalar case. An explicit bound on the number of interactions is given in \cite[p. 71-72]{HR}. The $2$-waves, associated to the linearly degenerate eigenvalue $a(u_\nu)$, never interact together since they are contact discontinuities.
 Due to the transversality assumption (USH), a $2$-wave can interact only once with  a $1$-wave and creates  a new $2$-wave (and does not modify the $1$-wave). Thus, the number of such interactions and  of $2$-waves is finite. This proves that the WFT is well defined for all time.   
\bigskip

Now, the approximate Riemann solver and the nonlinear interactions of the waves are detailed.

\subsubsection*{Approximate Riemann solver} \label{ss:approx-solver}

In this short section, the approximate Riemann solver is detailed.
Let $(u_-,Z_-)$, $(u_+,Z_+)$ denote the left and right states in the initial data.
 For $u_\nu$  the solution  is a series of entropic jumps  $u_{-}=u_0 <u_1<\ldots<u_m=u_{+}$ for the piecewise linear flux $f_\nu$.    For $Z_\nu$, there are many possibilities.  We want to approximate the exact solution of the Riemann problem and  keep the cubic estimates (Proposition \ref{pro:cub}) which generalize the Lax cubic estimates for genuinely nonlinear waves.   {\color{black}{So, we use the Rankine-Hugoniot curve to determine $Z_{m-1}, \ldots, Z_1$.}} The states    
$Z_i$, $i=m,m-1,\ldots, 2$ are built as follows:
 \begin{itemize}
 \item {If there is a jump between $u_{i-1}$ and $u_i$ corresponding to a shock or a contact discontinuity for $f_{\nu} $, then $Z_{i-1}$ is given by the Rankine-Hugoniot curve \eqref{def:RH},
 \begin{equation}
   Z_{i-1}=  \mathcal{RH}_+(u_{i-1};u_i,Z_{i}).
   \label{Z-WFT}
 \end{equation}
Note that $\mathcal{RH}^{\nu}_{+}=\mathcal{RH}_{+} $ on the grid $\{ (u,v) \in (\nu^{-1}\Z )\times \R \}$, where $\mathcal{RH}^{\nu}_{+}$ denotes the Rankine-Hugoniot curve corresponding to the approximate flux $F_{\nu}$. Thus, the jump  between $ (u_{i-1},Z_{i-1})$  and  $(u_{i},Z_{i})$ satisfies the Rankine-Hugoniot condition and there is no error in the weak formulation of the linear equation satisfied by $v$ and hence, no error in the weak formulation of the  approximate triangular system \eqref{eq:scl-nu}, \eqref{eq:lin}, where the exact flux $f$ is replaced by the piecewise-linear flux $f_\nu$ in \eqref{eq:scl}.}
\item  {If there is a 2-contact discontinuity (along which $u$ is constant and there is a jump between $v_\pm$), then the contact discontinuity is solved with the exact speed $a(u)$. Thus, again, there is no error of consistency  for the approximate system \eqref{eq:scl-nu}, \eqref{eq:lin}.}

 \end{itemize}

\subsubsection*{Nonlinear  wave interactions}    \label{ss:interactions}
We briefly describe the different possible nonlinear wave interactions and the details of the interactions are given in the $(u,Z)$ plane. 
{\color{black}{Here $Z$ denotes the 1-Riemann invariant with respect to the exact flux $F$.
\begin{lemma}\label{cubic-app}
The cubic estimates for $Z$ proved in  Proposition  \ref{pro:cub} also hold true across 1-waves (shock and 1-contact discontinuity) corresponding to the approximate flux   $F_{\nu}$. 
\end{lemma}
\begin{proof}
This immediately follows from \eqref{Z-WFT}, the fact that $\mathcal{RH}^{\nu}_{+}=\mathcal{RH}_{+} $ on the grid and the cubic estimates for $Z$ across 1-waves corresponding to $f$ (Proposition \ref{pro:cub}).
\end{proof}
}}
In all the cases, we consider three states $\tilde{U}_-,\tilde{U}_0,\tilde{U}_+$ before the interaction. Here $\tilde{U}_0$ denotes the intermediary state which disappears after the interaction.
The states $\tilde{U}_-,\tilde{U}_0$ are connected by an elementary wave and similarly the states
$\tilde{U}_0,\tilde{U}_+$ are connected by an elementary wave. An elementary wave is either a 1-wave: shock or a small jump, or a 2-wave: a contact discontinuity. 
At a time of interaction $t_{interact}$, the Riemann problem is solved with a new intermediary state $\tilde{U}_m$ (see Figure \ref{fig:Interaction}).  

\begin{figure}[H]
\includegraphics[clip, trim=50 400 100 100, scale=0.8]{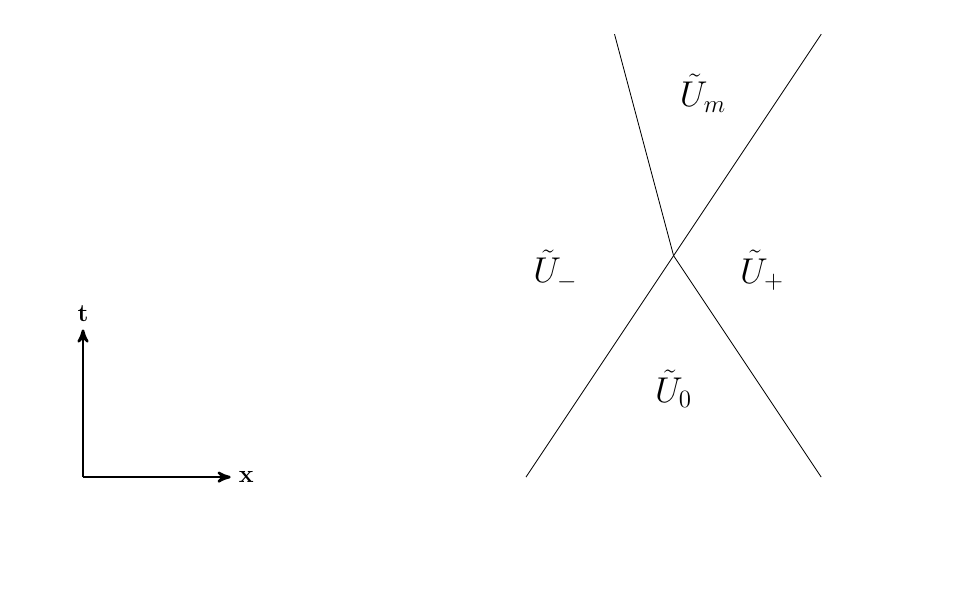}
\vspace*{-5mm}
\caption{Nonlinear interaction of two waves. The 1-wave on the left crosses the interaction point with the same speed and the same value $u_-$ on the left of the 1-wave ($u_m=u_-$) and $u_+$ on the right.  On the other hand, the second wave is affected by the interaction: the speed of the 2-wave and a new value $Z_m$ appears. }
\label{fig:Interaction}
\end{figure}
We shall use the following notations. 
\begin{itemize}
\item    $S$ or $S_1$  stands for  a shock wave which is always a 1-wave.
 
\item  $C_1$ or $C_2$ stands for a contact discontinuity associated to $\lambda_1$ or $\lambda_2$, a 1-wave or a 2-wave. 
$C_1$ is considered as a degenerate shock $S_1$. 
\end{itemize}
 The key point here is to understand the effect of the $L^\infty$ norm of $Z_{\nu}$ after the interaction. \\
 First of all, we note that there is no self interaction for the second family since there are only  $C_2$ waves. Also the case of interactions between the $1$- waves of the first family have already been well-studied \cite{Br00},  \cite{D16},\cite{HR} for the component $u_\nu$ and are not presented here.  The new feature is the effect of this interaction on the second component $Z_\nu$. We have seen in Lemma \ref{cubic-app} that the change in $L^{\infty} $ norm of $Z_{\nu}$ is of the order of the cube of the change in $u$ (or $u_{\nu} $).

Finally, we consider the interaction of a 1-wave with a  2-wave.
\begin{enumerate}
\item   In the case of an interaction of the form $ S_1-C_2$ (which means that a 1-shock interacts with a 2-contact discontinuity), the outgoing wave is of the form $C_{2}-S_{1}$.   \\   The shock continues with the same slope and the same value $u_-,u_+$ and still satisfies the entropy condition \eqref{Oleinik}. Thus, for $u$ there is no change before and after the interaction. Roughly speaking, the interaction of a 1-shock with a 2-contact discontinuity is transparent for $u$. 
On the contrary, there is a change for $Z_{\nu}$ and following \cite{HJ}, it can be shown that the change in $L^{\infty} $ norm of $Z_{\nu}$ is of the order of cube of the change in $u_{\nu}$.

\item   The interaction $C_1-C_2$ (that is, when a 1-contact discontinuity interacts with a 2-contact discontinuity)  generates outgoing waves of the form $C_2-C_1$. \\
This case can be dealt in a similar manner as in the case of $S_1-C_2 $ and it follows that the change in $L^{\infty} $ norm of $Z_{\nu}$ is of the order of cube of the change in $u_{\nu}$. 

\end{enumerate}

The reader interested by all the cases $S_1-C_1$, $C_1-S_1$, $S_1-C_2$, $C_1-C_2$ can consult \cite{HJ}  (and different notations where the $1$ and $2$ waves have to be exchanged) in the plane of Riemann invariant, a technique  used in \cite{Sm} for the isentropic Euler system.  Here we present the last two cases for the triangular system, the case  $S_1-C_2$ in figure \ref{fig:Interaction-S1-C2} and the case $C_1-C_2$ in figure \ref{fig:Interaction-C1-C2}. 

\begin{figure}
\includegraphics [scale=0.4]{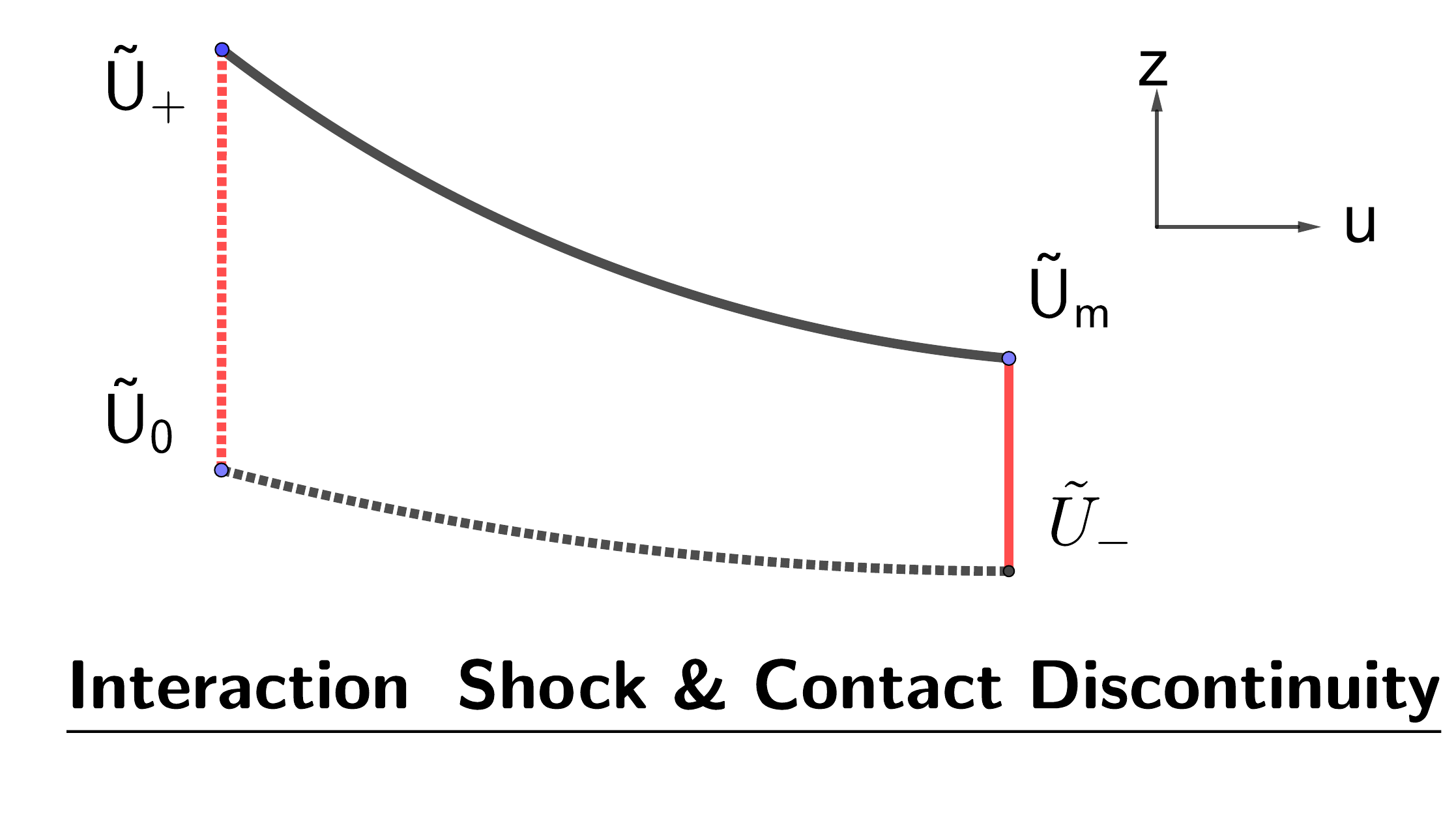}
\caption{Interaction of a  shock with a $2$-contact discontinuity.  The interacting waves  are represented by dotted lines,  a 1-wave in black  followed by a 2-wave in red. The full  lines represent the resulting waves.}
\label{fig:Interaction-S1-C2}
\end{figure}

\begin{figure}
\includegraphics [scale=0.4]{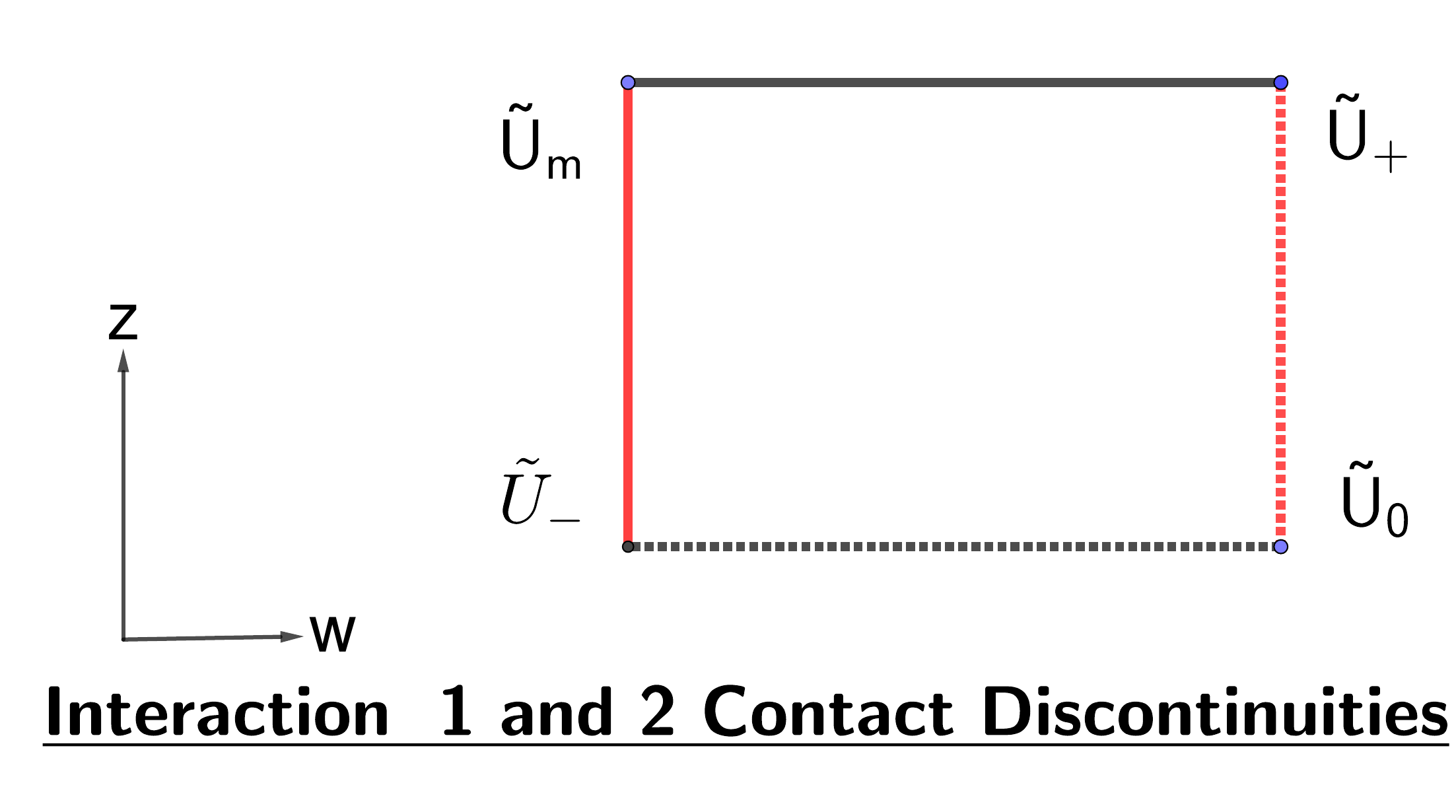}
\caption{Interaction between a $1$-contact discontinuity and a $2$-contact discontinuity.  The interacting waves  are represented by dotted lines,  a 1-wave in black (horizontal line)  followed by a 2-wave in red (vertical line). The full  lines represent the resulting waves.}
\label{fig:Interaction-C1-C2}
\end{figure}

\subsection{Uniform estimates}\label{uni-est}

The $BV^s$ estimates for $u_\nu$ are already known since $u$ is the entropy solution of the scalar conservation law \eqref{eq:scl}. These estimates are recalled briefly  in the first paragraph.
The only difficulty in this section is to obtain  the $L^\infty$ estimates for $v$. For this purpose, we generalize  the approach first used in \cite{BGJP} and recently in \cite{HJ}. The approach consists of bounding $v_\nu$,  indeed $Z_\nu$, along the $2$-characteristics.   For the chromatography system \cite{BGJP}, the $2$-characteristics are simply straight lines. 
  In general, here, the $2$-characteristics are piecewise linear. They are uniquely defined since the second eigenvalue $a(u)$  is linearly degenerate and its integral curves are  transverse to the discontinuity lines of the first field (assumption (USH)).  The precise definition of such characteristics is given in the second paragraph.  
Then the estimate on $Z_{\nu}$ along $2$-characteristics,  as in \cite{BGJP,HJ}, is given in the last paragraph using our generalized estimates on the Lax curves.

\subsubsection*{$BV^s$  estimates for $u$}  \label{ss:u-BVs}

The   $TV^s$ decay   is known for   the Glimm scheme, the Godunov scheme and the Wave Front Tracking algorithm \cite{BGJ6,BGJP}.   Another argument is that  $u_\nu$ is also  the exact  entropy solution of the scalar conservation law \eqref{eq:scl-nu} with the piecewise-constant initial data $u_{0,\nu}$, and the  decay of $TV^s u$ gives the uniform estimates with respect to $\nu$,
\begin{align}
TV^s u_\nu(\cdot,t)  \leq TV^s u_{0,\nu}  \leq TV^s u_0.
\end{align}

\subsubsection*{The approximate 2-characteristics}  \label{ss:approx-charac}
 
 Essentially, for the WFT, an approximate $i$-characteristic  is a continuous curve which is piecewise linear  following the velocity  $\lambda_i$, $i=1,2$. 
 Since the eigenvalues depend only on $u$, there is a problem to define an $i$-characteristic where $u$ is not defined.      For $i=2$, there is no problem of uniqueness, since a $2$-characteristic is always transverse to the discontinuity lines of $u$.   Thus, the $2$-characteristic crosses the $u$ discontinuity with a kink. \\
Let $\gamma_{\nu}(x_0,t) $ be the forward generalized $2$-characteristic starting at the point $x_{0} $,  that is $\gamma_{\nu}(x_0,t) $ is a solution of the differential inclusion
\begin{align*}
\frac{d}{dt}\gamma_{\nu}(x_0,t) &\in
 \left[ 
 a \left(u_{\nu}(\gamma_{\nu}(x_{0},t)-0,t)\right), 
 a \left(u_{\nu}(\gamma_{\nu}(x_{0},t)+0,t) \right)  \right ] , 
  && 
  \gamma_{\nu}(x_0,0)=x_{0}.
\end{align*}
 For the wave front tracking, these 2-characteristics are uniquely determined and  are piecewise linear continuous curves (thanks to the transversality assumption (USH)) and satisfy the differential equation 
 \begin{align}
 \frac{d}{dt}\gamma_{\nu}(x_0,t) & =  a\left(  u_{\nu} \left( \gamma_{\nu}(x_{0},t),t \right) \right)
 , &&  
   \gamma_{\nu}(x_0,0)=x_{0},
 \label{Z-char}
 \end{align}
except for a finite number of times which correspond to a jump of the piecewise constant function $u_\nu$.

\subsubsection*{$L^\infty$ estimates for $v$}  \label{ss:bound-v}

\begin{figure}
\vspace{-0.5cm}
\includegraphics[scale=0.8]{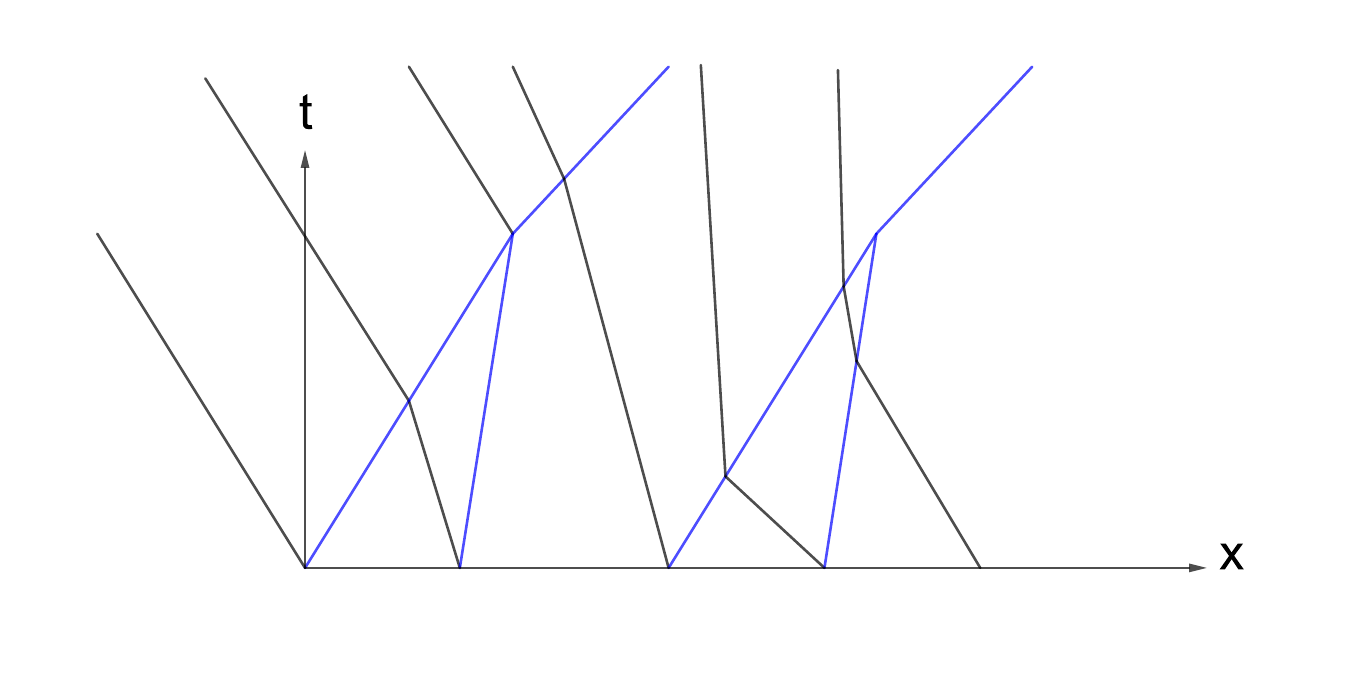}
\caption{Wave front tracking algorithm.  For the picture, $f'(u)>0 > a(u)$. Thus the 1-waves in blue go to the right and the 2-waves in black  to the left.  Notice that  these  1-waves are not affected by the interaction with the  2-waves, but  the 2-waves are affected by the interaction. 
}
\label{fig:WFT}
\end{figure}

The $L^\infty$ estimate on $v_\nu$ is  first obtained  on the approximate Riemann invariant $Z_\nu$.  
The key point is then to bound  $Z_\nu$  in $L^\infty$. 
 $Z_\nu$ is not constant  through a  1-wave: a shock wave or a 1-contact discontinuity.
But, we know that the variation of $Z_\nu$ is of order of the cube of the variation of $u_\nu$, by lemma \ref{cubic-app}.  
When  there is no 1-wave, the simple curve to bound $Z_\nu$ is the $2$-characteristic. In this case, $Z_\nu ( \gamma_\nu(x_0),t) = Z_\nu(x_0,0)$.   

Note that the $2$-characteristics starting at time $t=0 $ span the whole $(t,x) $ half-plane but $Z_{\nu}$ is only defined outside a finite number of $2$-characteristics.

For example, $Z_{\nu} $ is not well-defined on a $2$-characteristic that contains a $2$-wave (contact discontinuity for the second family). But these are only finite in number as the number of $2$-wave fronts is finite. 

Also $Z_{\nu}$ is not well-defined on $2$-characteristics passing through the points of intersections of waves corresponding to the first family. Again there are only finite number of such curves.

$Z_{\nu} $ is also not well-defined on a finite number of $2$-characteristics originating from the points on $t=0 $ where the Riemann problems are solved initially.

But since there are only finite number of such $2$-characteristics, this is enough to estimate $\|Z_\nu\|_\infty $. 

 Now, due to the transversality conditions,  a 2-characteristic meets many 1-waves. Thanks to lemma \ref{cubic-app}, we have a cubic estimate for the $L^\infty$ norm and the total variation of $Z_{\nu}$ through a 1-wave. More precisely, if $t_1$ is a time just before the 2-characteristic meets a 1-wave and $t_2$ is the time just after the 2-characteristic crosses the 1-wave, we have the following estimate,  
\begin{align*}
  \|Z_\nu \|_{\infty, ( \gamma_{\nu}(x_0,\cdot),\cdot)[t_1,t_2] }   
  & \leq |Z_\nu(t_1)| \exp  \left( \mathcal{O} \left(  TV^{1/3} u_\nu( \gamma_{\nu}(x_0,t_2),t_2)
   \right)\right), \\
  TVZ_\nu ( \gamma_{\nu}(x_0,\cdot),\cdot)[t_1,t_2]  & \leq    \|Z_\nu \|_{\infty, ( \gamma_{\nu}(x_0,\cdot),\cdot)[t_1,t_2] } 
\left( \mathcal{O}( TV^{1/3}u_\nu  ( \gamma_{\nu}(x_0,\cdot),\cdot)[t_1,t_2] )\right).
\end{align*}
 That means that  the $L^\infty$ norm and the  total variation of $Z_\nu$ 
  along the piece of curve $\{ ( \gamma_{\nu}(x_0,t),t),\, t \in [t_1,t_2] \}$ is controlled by 
  $TV^{1/3}u_\nu$  along the same piece of curve.
  Notice also that the sign of $Z_\nu$ is constant along a 2-characteristic by lemma \ref{lem:flatRH}.

  First, the $L^\infty$ norm of $Z$ is bounded on the 2-characteristic and, second, the estimate on $TV \, Z_\nu$ follows.
The total variation  is additive  and the fractional total variation is sub-additive \cite{Bruneau74}. So  adding all these estimates, on the whole 2-characteristic $ \Gamma(x_0)= \{ ( \gamma_{\nu}(x_0,t),t),\, t > 0 \}$ starting at  $x=x_0$,
 we have the  estimates
  \begin{align*}
  \| Z\|_{\infty, \Gamma(x_0) } & \leq |Z_{0,\nu}(x_0)| \exp \left( \mathcal{O} \left(  TV^{1/3} u_\nu  \left[\Gamma(x_0) \right] \right) \right), \\
           &  \leq \|Z_{0}\|_\infty \exp \left( \mathcal{O} \left(  TV^{1/3} u_\nu \left[\Gamma(x_0) \right] \right) \right), \\
  TVZ_\nu \left[\Gamma(x_0) \right]  & \leq  
    \|Z_{0}\|_\infty \Psi \left( \mathcal{O} \left(  TV^{1/3} u_\nu \left[\Gamma(x_0) \right] \right) \right)
 ,  \\
  \Psi(x) &=  x \exp(x).
\end{align*}
Now note that $\Gamma(x_0)$  is a space like curve for the 1-characteristics of \eqref{eq:scl-nu} (\cite{D16}). Hence 
\begin{align*}
 TV^{1/3}u_\nu \left[\Gamma(x_0) \right]  \leq  TV^{1/3}u_{0,\nu} \leq   TV^{1/3}u_{0}. 
\end{align*}
This follows from the fact the wave-fans corresponding to $u_{\nu} $ are monotone with respect to the left and right end states. Therefore no extremal values are added when we measure the fractional variation along $\Gamma(x_{0}) $.\\
Together these yield $L^\infty$ and $BV$ bounds for $Z_\nu$ (uniform with respect to $\nu$) along the 2-characteristics, 
\begin{align}
    \|Z_\nu \|_\infty   & \leq 
         \|Z_{0}\|_\infty  \exp \left( \mathcal{O}(1) TV^{1/3}u_{0} \right) , \\
     TVZ_\nu \left[\Gamma(x_0) \right]  & \leq  
    \|Z_{0}\|_\infty \Psi \left( \mathcal{O} \left(  TV^{1/3} u_0\right) \right) . 
\end{align}

Notice also that the positivity is preserved.  If $\inf v_0 > 0 $, then $ \inf Z_0 > 0$ and by similar arguments presented above, we have $ \inf Z_{\nu} > \inf Z_0   \exp \left( \mathcal{O}(1) TV^{1/3}u_{0} \right)   > 0$, where  the constant $\mathcal{O}(1)$ is negative. 
Finally,   a   $L^\infty$ bound (respectively a positivity) for $Z_\nu$    yields    a   $L^\infty$ bound (respectively a positivity) for $v_\nu$.
\medskip

For the triangular system,  the $L^\infty$ bound of $Z_\nu$ and hence of $v_\nu$  is enough to pass to the weak  limit in \eqref{eq:lin} since  the left hand side is linear with respect to  $v_\nu$.   
The  uniform $BV$ bound on $Z_{\nu}$  along the 2-characteristics can be used to recover a strong trace at $t=0$, like in \cite{BGJ2}  (at $x=0$).

\subsection{Passage to the limit in the weak formulation}    \label{ss: weak sol}
Passing to the  strong-weak limit  in the equation \eqref{eq:lin} which is linear with respect to $v$  allows us to get a weak solution. But to do so, we need to understand the error of consistency of the scheme.  There is no error of consistency in the independent scalar equation for $u_\nu$ \eqref{eq:scl-nu} \cite{D72}.  By construction,  our approximate solver  does not produce error for the linear equation \eqref{eq:lin}.  That means that the approximate solution $(u_\nu, v_\nu)$ is an exact solution of the approximate system  \eqref{eq:scl-nu},\eqref{eq:lin}:
 \begin{align*}
 \partial_t u_\nu   &+  \partial_x f_\nu (u_\nu)    =  0 ,
  \\
  \partial_t v_\nu  &+   \partial_x \left ( a (u_\nu) v_\nu \right )   =   0.  
  \end{align*} 
  Therefore as in \cite{D72}, one can pass to the limit and obtain a weak solution of the exact system \eqref{eq:scl},\eqref{eq:lin}.

\section{
Blow-up of  an entropy  solution with $u_0 \notin BV^{1/3}$
  } \label{sec:boum}

 In this section, we provide a proof of Theorem \ref{blow}. 
 For this purpose, we construct initial data $u_{0}, Z_{0}$ that satisfy
\begin{equation}
\begin{cases}
\begin{aligned}
 u_{0} &\in BV^{1/3-0}(\mathbb{R}) ,\\
Z_{0} &\in L^{\infty}(\mathbb{R}),
\end{aligned}
\end{cases}
\notag
\end{equation}
and exhibit blow-up in $L^{\infty} $ norm of $Z$ at $t=0+ $ for the system \eqref{eq:scl}-\eqref{eq:lin}.\\
Recall that the notation $u_{0} \in BV^{1/3-0}(\mathbb{R})$  means that $u_0$ is in $BV^s$ for all $s \in [0,1/3)$.
Our idea of construction is motivated by similar examples studied in \cite{AGV1, CJ2, DW,GGJJ}.
\\
We would like to emphasize that the solution proposed below is not based on the Wave-Front Tracking algorithm discussed earlier, rather it is an exact solution.
\smallskip

Notice that when $v_0 \equiv 0$, i.e.  $Z_0 \equiv 0$, no blow-up occurs  since $(u,v)\equiv(u, 0)$ gives a global entropy solution   where $u$ the entropy solution associated to the $L^\infty$ initial data  $u_0$.   Thus, in the following construction, we have to avoid the value $0$ for $Z$.

Let us consider the system \eqref{eq:scl}-\eqref{eq:lin} with flux $f(u)=\frac{u^2}{2} $, $a(u)=u-1 $ and initial data $u_0$ as described below and  $Z_0 \equiv 1 $.\\
Let $x_0=0 $ and $x_n=1-\frac{1}{2^n}, \ n\geq 1 $. Let $B_n$ be chosen such that
\[
 x_n=x_{n-1}+2 B_n,
 \]
that is, $B_n=\frac{1}{2^{n+1}} $. Let $b_n=\frac{1}{(n+26)^{\frac{1}{3}}} $ and we define
\[U_{0}(x,b,B)=b \mathbbm{1}_{[0,B)}(x)-b \mathbbm{1}_{[B,2B)}(x). \]
Using this we define the initial data $u_0$ as
\[u_{0}(x)=\sum_{n \geq 1} U_{0}(x-x_{n-1},b_n,B_n)  .\]
Note that $\|u_{0} \|_{\infty} \leq \frac{1}{3} $ and therefore, the condition of uniform strict hyperbolicity is satisfied.\\
\begin{figure}
   \centering
   \vspace*{-1cm}
    \includegraphics[scale=0.8]{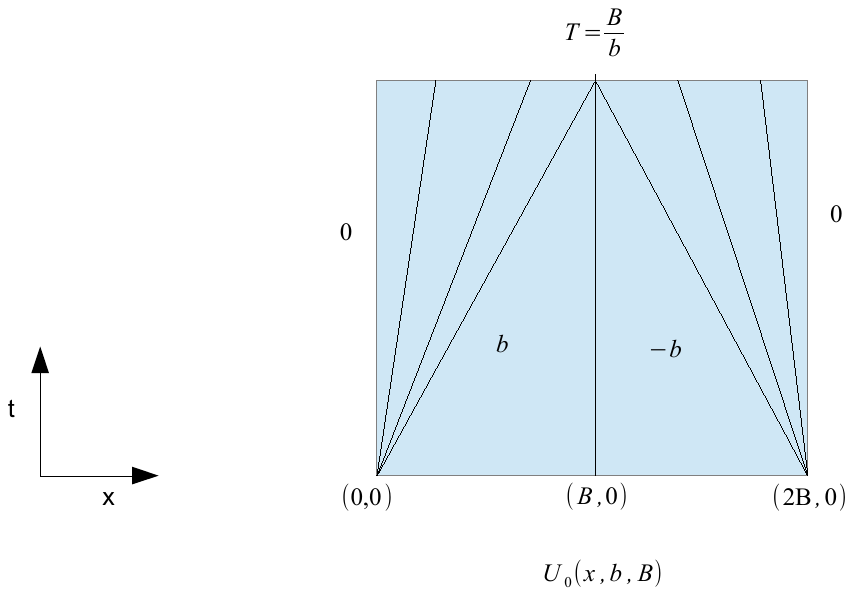}
      \vspace*{-150mm}
    \caption{A typical building-block}
\end{figure}
The first interaction times $T_n $ for the Riemann problems for the equation 
\[\partial_{t}u+\partial_{x} f(u)=0 \]
with initial data $u_0$ is given by \[T_{n}=\frac{B_n}{b_n}=\frac{(n+26)^{\frac{1}{3}}}{2^{n+1}} .\]
Note that the first interaction time $T_{n}$ satisfies the relation 
\begin{equation}
T_{n}>1-x_{n} .
\label{int-time}
\end{equation}
The initial data $u_0$, described above, clearly does not belong to $BV^{\frac{1}{3}}(\mathbb{R})$, but, as
 in \cite{BGJP,CJ2} 
  we can conclude that $u_{0} \in BV^{1/3-0}(\mathbb{R})$.\\ 
We use the forward generalized characteristic for $Z$.  Since $a(u)=u-1 $ and the first interaction times of the Riemann problems for $u$ satisfy \eqref{int-time}, it follows that the forward generalized characteristic for $Z$ starting at the point $x_{\infty}:=1$ crosses infinitely many shocks before the first interaction of the waves in $u$.\\
As we have already seen from the nonlinear wave interactions, the $L^{\infty}$ norm of $Z$ does not change when it interacts with a $1-$rarefaction wave. 

Now, let us consider a left state $u_{-} $ and a right state $u_{+}$ connected by a 1-shock wave. In our example, we have $u_{-}=b>0 $ and $u_{+}=-b<0 $. Hence, the speed of the shock $s=0$. Also by construction $0< b_{n}<\frac{1}{2}<1 $ and hence we assume that the prototype $b $ satisfies the same.\\
Now 
\begin{equation}
\begin{aligned}
\frac{Z_{-}}{Z_{+}}&= \frac{s-a(u_{+})}{s-a(u_{-})} \exp \left[A(u_{-})-A(u_{+}) \right] \\
&=\frac{a(u_{+})}{a(u_{-})} \exp \left[A(u_{-})-A(u_{+}) \right]\\
&=\frac{u_{+}-1}{u_{-}-1} \exp \left[A(u_{-})-A(u_{+}) \right]\\
&=\frac{1+b}{1-b} \exp \left[A(u_{-})-A(u_{+}) \right].
\end{aligned}
\label{equation-200}
\end{equation}
Also 
\begin{equation}
\begin{aligned}
A'(u)&=\frac{a'(u)}{a(u)-f'(u)}=\frac{1}{(u-1)-u}=-1<0,
\end{aligned}
\notag
\end{equation}
and hence 
\begin{equation}
A(u_{-})-A(u_{+})=\int_{u_{+}}^{u_{-}} A'(u)\ du=-(u_{-}-u_{+})=-2b.
\notag
\end{equation}
Therefore from \eqref{equation-200}, we have 
\begin{equation}
\frac{Z_{-}}{Z_{+}}=\frac{1+b}{1-b} \ e^{-2b}.
\label{equation-201}
\end{equation}
We show that for $b$ positive and small enough,  
\begin{align} \label{ineq:b}
 \frac{1+b}{1-b} \ e^{-2b}>1,
\end{align} and therefore $Z_{-}>Z_{+} $. Thus, $Z$ increases in strength as the forward 2-generalized characteristic crosses a shock (from right to left).\\

Notice that  we simply need  that Z increases when   $b  \sim 0 $.   It is for very small oscillations that Z blows-up. Inequality  \eqref{ineq:b} follows from  a Taylor expansion up to the third order:
\begin{align*}
 (1+b)  \exp(-2b) & = (1+b)(1 -2b+ 1/2 \cdot (2b)^2 -1/6 \cdot (2b)^3 + \mathcal{O}(b)^4) \\
   & =  1 + (1-2)  b +   (-2 +2)     b^2  + ( 2- 4/3)           b^3 +   \mathcal{O}(b)^4 \\
  &=   1 -b +  2/3 b^3 + \mathcal{O}(b)^4   \\
  & >    1 -b, 
\end{align*}
 for b sufficiently small  and hence $Z_-> Z_+> 0$.\\

Thus, due to an interaction with a $1-$shock wave, there is a change of order $[u]^3 $ in the $L^{\infty} $ norm of $Z$.
Since 
\begin{equation}
\sum \vert [u] ^3 \vert=\sum_{n\geq 1} \left(\frac{2}{(n+26)^{\frac{1}{3}}} \right)^{3}=8\sum_{n \geq 1} \frac{1}{(n+26)}=+\infty, 
\notag
\end{equation}
we find that the $L^{\infty} $ norm for $Z$ blows up.
%

\begin{remark}  Let $\Gamma(x_0)= \{ (\gamma(x_0,t), t),\, t \geq 0\}$ be the  2-characteristic issued from $x=x_0$ at $t=0$. 
The solution is well defined  under   $\Gamma(1)$ that is on  the set 
$\{ (x,t),  x < \gamma(1,t), t  \geq 0  \}$.  But, over    $\Gamma(1)$, $\{ (x,t),  1 > x  > \gamma(1,t), t  > 0  \}$, $Z$ and $v$ blow up, $v=+\infty$. 
\end{remark}

\begin{remark}
This example does not contradict the Lax-Oleinik smoothing effect (\cite{Lax57}) as the blow-up for $Z$ occurs only at $t=0$, that is, there is an immediate blow-up. Such a blow-up is not possible for a time $t_{0}>0 $ due to the $BV$ smoothing of $u$. 
\end{remark}

Now suppose that there exists a bounded weak solution for the system \eqref{eq:scl}-\eqref{eq:lin} with the above initial data. Then since the solution of the Riemann problem in the class of bounded weak solutions is unique upto the time of first interaction of the waves (see appendix \ref{uniq-riem}), any bounded weak solution will satisfy the construction discussed above. This, in turn, implies that the bounded weak solution exhibits a blow-up in $L^{\infty} $ norm, which contradicts its definition.\\
Thus, there cannot exist a bounded weak solution for the system \eqref{eq:scl}-\eqref{eq:lin} with the above initial data, which proves Theorem \ref{blow} . This also shows that the existence Theorem \ref{main} is optimal in the class of bounded weak solutions.

 \section*{Acknowledgments}     
   A part of this work was done when the last author was visiting  NISER, Bhubaneswar, India.  
   We  also thank the IFCAM project ``$BV^s$, control \& interfaces''  which allowed the interaction  between the Indian hyperbolic community with the French one. We also thank  Boris Haspot for fruitful discussions.

\appendix

\section{On Lipschitz test functions}\label{lipschitz}
Here we briefly discuss the  use of Lipschitz test functions instead of $C^1$ test functions   for bounded weak entropy  solutions for triangular systems \eqref{eq:scl}-\eqref{eq:lin}. We show that if $u,v \in L^{\infty} $ is a weak bounded solution (and therefore the solution constructed in Section \ref{sec:exist}) for the triangular system, then the $C^1$ test functions can be enlargerd to Lipschitz functions. 
This can be proved using the following observations. Note that $(u,v) $ is a weak solution and hence satisfies the integral identities \eqref{int-identity-1}-\eqref{int-identity-2} for compactly supported smooth test functions. Now if we consider a compactly supported Lipschitz continuous test function $\psi$, we can construct (using a standard mollifier) a sequence of compactly supported smooth test functions $\{\phi_{n}\} $ that converges pointwise almost everywhere to $\psi $. \\
An application of dominated convergence theorem shows that $\partial_{t}\phi_{n} $ and $\partial_{x}\phi_{n} $ converge to $\partial_{t}\psi $ and $\partial_{x}\psi $ respectively at points of continuity (and therefore almost everywhere) of $\psi$. \\
Since the integral identities \eqref{int-identity-1}-\eqref{int-identity-2} are satisfied for each $\phi_{n} $, we can use the above facts to pass to the limit in the identities to prove the integral identities for the Lipschitz continuous test function $\psi $.

\section{Uniqueness of Riemann problem in $L^\infty$} 
 \label{uniq-riem}
We next show that in the case of strictly convex fluxes $f$, the solution of the Riemann problem is unique in the class of bounded entropy  solutions at least upto the time of first interaction of waves.
In this direction, we first note the following result.
\begin{lemma}
If $(u,v)$ is a bounded entropy solution to the system \eqref{eq:scl}-\eqref{eq:lin} with initial data $(u_{0},v_{0})\in L^{\infty}(\mathbb{R}) \times L^{\infty}(\mathbb{R}) $, 
then $Z:=v\ \textnormal{exp}(A(u))$ satisfies
\begin{equation}
\partial_{t}Z+a(u) \partial_{x}Z\stackrel{E}{=}0, 
\label{Z-identity}
\end{equation}
in open sets where $u$ is Lipschitz continuous.\\
 Here $E:= Lip_{c}'(\R \times [0,+\infty[, \R) $ denotes the dual of the space of compactly supported Lipschitz continuous functions.
\end{lemma}
\begin{proof}
Since $u $ is Lipschitz continuous and $A\in C^{1}(\mathbb{R}) $, we can write any $\psi \in Lip_{c} $ in the form $\psi=\phi\ \text{exp}(-A(u)) $ for a suitable $\phi \in Lip_{c} $.
 Therefore, for $\psi \in Lip_{c}(\R \times [0,+\infty[, \R) $, we can write,
 \begin{equation}
 \begin{aligned}
 &-\int_{0}^{\infty} \int_{\mathbb{R}} \left[Z \partial_{t}\psi+Z \partial_{x}(a(u)\psi) \right]\ dx dt+\int_{\mathbb{R}} Z_{0}(x) \psi(x,0)\ dx \\
 &=-\int_{0}^{\infty} \int_{\mathbb{R}} \left[Z\ \partial_{t}\left(\phi\ \text{exp}(-A(u)) \right)+Z\ \partial_{x}\left(a(u)\phi\ \text{exp}(-A(u))\right) \right]\ dx dt\\ 
 &\qquad +\int_{\mathbb{R}} Z_{0}(x) \phi(x,0)\ \text{exp}(-A(u_{0}))\ dx \\
 &=-\int_{0}^{\infty} \int_{\mathbb{R}} \left[Z\ \text{exp}(-A(u))\ \partial_{t}\phi+Z\ a(u)\ \text{exp}(-A(u))\ \partial_{x}\phi \right]\ dx dt\\ 
 &\qquad +\int_{\mathbb{R}} Z_{0}(x) \phi(x,0)\ \text{exp}(-A(u_{0}))\ dx \\
 &\qquad +\int_{0}^{\infty} \int_{\mathbb{R}} \left[Z A'(u)\partial_{t}u \ \text{exp}(-A(u)) \phi+Z A'(u)\partial_{x}u \ \phi a(u) \text{exp}(-A(u))\right]\ dx dt\\
 &\qquad -\int_{0}^{\infty} \int_{\mathbb{R}} Z \ \text{exp}(-A(u))\phi a'(u)\partial_{x}u\ dx dt\\
 &=\int_{0}^{\infty} \int_{\mathbb{R}} Z\phi\ \text{exp}(-A(u)) \left(A'(u)\partial_{t}u+A'(u)a(u)\partial_{x}u-a'(u)\partial_{x}u \right)\ dx dt,
 \end{aligned}
 \label{z1}
 \end{equation}
 where we use the fact that $v$ satisfies the integral identity \eqref{int-identity-2} for any $\phi \in Lip_{c}(\R \times [0,+\infty[, \R) $.\\
 Now $\partial_{t}u $ and $\partial_{x}f(u) $ belong to $L^{\infty} $ and hence in $L^{1}_{loc} $, and also satisfy the integral identity \eqref{int-identity-1} for all compactly supported smooth functions. Therefore $\partial_{t}u=-\partial_{x}f(u) $ almost everywhere. Also using the fact that $f \in C^{1} $ and $u$ is continuous, using Volpert's formula \cite{Volpert}, it follows that $-\partial_{x}f(u)=f'(u)\partial_{x}u $ almost everywhere. \\
 Using these facts and the definition of $A$ in \eqref{z1}, it follows that
 \begin{equation}
 -\int_{0}^{\infty} \int_{\mathbb{R}} \left[Z \partial_{t}\psi+Z \partial_{x}(a(u)\psi) \right]\ dx dt+\int_{\mathbb{R}} Z_{0}(x) \psi(x,0)\ dx=0,
 \notag
 \end{equation}
 for $\psi \in Lip_{c}(\R \times [0,+\infty[, \R) $.
\end{proof}
Now suppose that $f$ is strictly convex and consider the Riemann problem with left and right states given by $(u_{-},v_{-}) $ and $(u_{+},v_{+}) $ respectively. Let the corresponding states in $(u,Z) $ coordinates be given by $ (u_{-},Z_{-})$ and $(u_{+},Z_{+}) $ respectively. Without loss of generality, we assume that the Riemann problem is based at the point $(0,0) $.

First we consider the case when there is a rarefaction wave in $u$. Note that in this case $u$ is a Lipschitz continuous function away from a ball $B(0,r) $ of radius $r$ around $(0,0) $, for any $r>0$. \\
From the previous lemma, we conclude that away from the ball $B(0,r) $, $Z$ (corresponding to any bounded weak solution $(u,v) $) satisfies the identity \eqref{Z-identity} and by unicity of $u$ and the unicity of weak solutions of a transport equation with Lipschitz continuous velocity field, $Z$ is uniquely determined.\\
Since this is true for any $r>0$, it follows that in the class of bounded weak solutions, $(u,Z) $ (and therefore $(u,v) $) is unique away from the set $\{x=a(u_{-})t \} $ of measure zero.   

Next let us consider the case when there is a shock wave in $u$. In this case, using the previous lemma, we see that $Z$ (corresponding to any bounded weak solution $ (u,v)$) satisfies the transport equation \eqref{Z-identity} with velocity $a(u_{-})$ in the region 
\[\left\{(x,t): -\infty < \frac{x}{t} <s:=\frac{[f(u)]}{[u]} \right\},\] 
and therefore is uniquely determined in this region and has a trace on the line $x=st $.\\ Similarly, $Z$ satisfies the transport equation \eqref{Z-identity} with velocity $a(u_{+})$ in the region 
\[\left\{(x,t): s < \frac{x}{t} < +\infty\right\},\] 
and therefore is uniquely determined in this region as well.\\
But any bounded weak solution $(u,v) $ satisfies the same Rankine-Hugoniot condition across the shock curve, and therefore, the left trace of $Z$ on the line  $x=st $ is uniquely determined by the states $u_{-},u_{+},v_{+} $, which implies the unicity of $Z$ as well.\\
Therefore,  for a strictly convex flux $f$, at least upto the time of first interaction of waves, the solution of the Riemann problem is unique in the class of bounded entropy solutions. 

Note that the above proof holds for initial data $u_{0}$ which gives rise to entropy solutions $u$ that are locally Lipschitz except for a finite number of lines, and hence applies to the initial data that we construct in Section \ref{sec:boum}.
Uniqueness for the slightly smaller class of picewise $C^1$ entropy solutions  are well known   \cite{Lax book} and has been already obtained for the (PSA) system \cite{BGJ1}. 

Notice also that  for the Kruzkov solution  $u$ a continuity in time with respect to $L^1_{loc}$ is space is required \cite{K}, but for a nonlinear flux this property is automatically fullfilled  \cite{CR}.
Notice that since $u_0$ belongs to $BV^s$, $s>0$, the solution is continuous in time with values in $L^1_{loc}$ in space \cite{BGJ6,JR1}. Thus, the initial data is recovered.

\bibliographystyle{plain}

\date{\today}

\end{document}